\documentclass[reqno,10pt]{amsart}
\usepackage[english]{babel}
\usepackage{amsmath}
\newcommand  {\Wert}{{|\hspace{-0.03cm} \Vert }}
\usepackage{amssymb}
\usepackage[scr=boondoxo]{mathalfa}

\newcommand{\bpr}{\begin{trivlist} \item[]{\bf Proof. }}
\newcommand{\epr}{\hspace*{\fill} $\qed$\end{trivlist}}

\renewcommand{\i}{{\mathrm i}}
\newcommand{\e}{{\mathrm e}}

\newcommand{\be}{\begin{eqnarray}}
\newcommand{\ee}{\end{eqnarray}}
\newcommand{\ba}{\begin{align}}
\newcommand{\ea}{\end{align}}
\newcommand{\bi}{\begin{itemize}}
\newcommand{\ei}{\end{itemize}}

\newcommand{\sU}{\mathscr U}
\newcommand{\sW}{\mathscr W}
\newcommand{\sC}{\mathscr C}
\newcommand{\spec}{\mathrm{spec}}
\newcommand{\id}{\mathrm{id}}
\newcommand{\kE}{\mathcal E}
\newcommand{\kF}{\mathcal F}
\newcommand{\kC}{\mathcal C}
\newcommand{\kZ}{\mathcal Z}
\newcommand{\kX}{\mathcal X}
\newcommand{\kY}{\mathcal Y}
\newcommand{\kL}{\mathcal L}
\newcommand{\kT}{\mathcal T}
\newcommand{\kM}{\mathcal M}
\newcommand{\kO}{\mathcal O}
\newcommand{\kB}{\mathcal B}
\newcommand{\kS}{\mathcal S}
\newcommand{\kV}{\mathcal V}

\newcommand{\kW}{\mathcal W}
\newcommand{\kH}{\mathcal H}

\renewcommand{\S}{\mathbb S}
\renewcommand{\P}{\mathbb P}
\newcommand{\N}{\mathbb N}
\newcommand{\C}{\mathbb C}

\newcommand{\R}{\mathbb R}
\newcommand{\Z}{\mathbb Z}

 \renewcommand{\d}{{\mathrm d}}
\newtheorem{theorem}{Theorem}[section]

\newtheorem{lemma}[theorem]{Lemma}
\newtheorem{cor}[theorem]{Corollary}
\newtheorem{example}[theorem]{Example}
\newtheorem{remark}[theorem]{Remark}

\numberwithin{equation}{section}

\begin{document}

\title {Exponential estimates of symplectic slow manifolds }

\author {K. U. Kristiansen and C. Wulff} 
\date\today
\maketitle

\vspace* {-2em}
\begin{center}
\begin{tabular}{c}
Department of Mathematics, \\
Technical University of Denmark, \\
2800 Kgs. Lyngby, \\
DK \\
and\\
Department of Mathematics, \\
University of Surrey, \\
Guildford, GU2 7XH\\
UK
\end{tabular}
\end{center}
 \begin{abstract}
In this paper we prove the existence of an almost invariant 
symplectic slow manifold for  analytic Hamiltonian  slow-fast systems with finitely many slow degrees of freedom for which the error field is exponentially small.
 We allow for infinitely many fast degrees of freedom.  
  The method we use is motivated by 
  a paper of MacKay from 2004. The method does not notice resonances, 
  and therefore we do not pose any restrictions on the motion normal to the 
  slow manifold other than it being fast and analytic.  We also present a stability result  and obtain a generalization of a result  of Gelfreich and Lerman on an invariant slow manifold to (finitely) many  fast degrees of freedom. 
\end{abstract}

\tableofcontents


\section{Introduction}\label{s.intro}
 Singularly perturbed systems involving different time and/or space scales arise in a wide variety of 
scientific problems. Important examples include: meteorology and short-term weather forecasting 
\cite{lor2,lor1,tem1}, molecular physics and the Born-Oppenheimer approximation \cite{McQ1}, 
chemical enzyme kinetics and the Michaelis-Menten mechanism \cite{AFehrst,micmen1}, combustion \cite{UMaas}, and the evolution and stability of the solar system \cite{las1,las2}. 
The main advantage 
of identifying slow and fast variables is dimension reduction by which all the fast variables are 
``slaved'' to the slow ones through the \textit{slow manifold}. In many dissipative systems, the rigorous foundation for a reduction to an attracting, invariant, lower dimensional manifold is provided by Fenichel's theory \cite{fen1,fen2}.
In contrast, the conservative cases, exemplified by models of nonlinear 
   waves  \cite{AK,lor1,lor2,mac2,van1} and of tethered satellites \cite{kri1,kri2,kri4}, with high frequency oscillations, do not support such a general theory.
Nevertheless, such systems may possess \textit{almost invariant slow manifolds}, see
\cite{geller1,Lu,mac1,nei87,van1},
  on which the angle between the vector field and the tangent space is exponentially small (with respect to the time-scale separation). Orbits of the system may therefore spend a significant amount of time near such a slow manifold. 
  
  Dimension reduction is one of the main 
aims and tools for a dynamicist and the elimination of fast variables is very useful in, for example,
numerical computations. In fact, since fast variables require more computational effort and evaluations,
 this reduction often bridges the gap between tractable and intractable computations. An example of this
  is the long time ($Gyear$s) integration of the solar system, see \cite{las1,las2}. See also 
  \cite{cotrei1,kap1,kap2} 
for a numerical treatment of slow-fast systems.

  In this paper, we 
prove the existence of an exponentially accurate  symplectic slow manifold for  a  general class of analytic slow-fast Hamiltonian systems. Here we allow the fast sub-system to be a semilinear evolution PDE. We also present a stability result, and obtain a generalization of a result in \cite{geller2} on an invariant slow manifold to (finitely) many  fast degrees of freedom.

\subsection{Slow-fast systems.}
Consider the system
\begin{align}
 \dot w &=\epsilon   W(w, z),\quad \dot   z =   Z(w,  z),\label{e.xys0}
\end{align}
where $\dot{()}=\frac{\d}{\d t}$ and 
with $\epsilon$ a small parameter. The analytic vector-fields $  W$ and $  Z$ will in general also 
depend upon $\epsilon$, but we shall suppress this dependency throughout.

Taylor-expanding the second equation of \eqref{e.xys0} around $z=0$ gives 
\begin{align}
 \dot w =\epsilon W(w,z),\quad \dot z = {r}(w) + A(w)z+{F}(w,z),\label{e.xys1}
\end{align}
where $F=\kO(\|z\|^2)$. Note that if $r\equiv 0$ then the  manifold
 $\kM= \{z=0\} $ is invariant. 
Therefore  $\kM= \{z=0\} $ is called  a slow manifold in the sense of MacKay 
\cite[Definition 1]{mac1} if $\|r\|$ is small 
and  
  $\Vert (\partial Z/\partial z)(w,0)^{-1}\Vert = \|A(w)^{-1}\|$ exists. This ensures that
     $\kM$ is close to being invariant, that the normal motion is truly fast, and that we can solve the equation $\dot z = 0$ for $z= \zeta(w)$. The latter property 
     allows us to compute an approximately invariant slow manifold.
   In particular, if $A(w)$ is  elliptic or hyperbolic 
      then $\kM$ is said to be normally elliptic respectively hyperbolic.
One of the main tasks in singular perturbation theory is to determine the fate of $\kM$ for 
$\epsilon>0$ but small. When
 $\kM$ is  normally hyperbolic then there exists a perturbed, invariant slow manifold $\kM(\epsilon)$
 for $\epsilon\ne 0$ nearby \cite{fen1,fen2}. 
The invariance condition is, however, too much to aim for in the normally elliptic setting because normally elliptic manifolds are unlikely 
 to persist under  typical perturbations \cite{mac1,man1}. 
 However the normally elliptic setting is an important case in Hamiltonian systems which are our main point of interest in this paper. 
 In this case, the important  question is how to 
 improve the accuracy of  approximately invariant slow manifolds
 and to determine how long 
 orbits of the full system remain close to those of the approximate dynamics on the slow manifold. 
 
\subsection{Improving a slow manifold}
 MacKay, still in \cite{mac1}, presented a method for improving a given slow manifold and he sketched how, when applied successively to an analytic system of  differential equations, this could lead to an exponentially accurate slow manifold $\kM$ in the sense that the error field is $\|r\|=\kO(e^{-c/\epsilon})$.
The method of MacKay  for improving a slow manifold does not separate 
normally hyperbolic from normally elliptic slow manifolds.
Moreover MacKay allowed for an unbounded fast vector field and argued that one of the advantages of his method is that the improved slow manifold contained all nearby equilibria and conjectured that the error in fact behaved like 
 $\kO((\epsilon \Vert W(w,\cdot)\Vert)^n)$, for any $n$, for analytic systems eventually leading to $\kO(e^{-c/(\epsilon \Vert W(w,\cdot)\Vert)})$, with $W$ vanishing on equilibria. 
 In the appendix we provide a simple counterexample  (Example \ref{ex.e1}) to this last statement for $n>1$.
 
The method of MacKay was previously presented by Fraser and Roussel in \cite{fra1, Roussel}
 in a slightly different form. In the reduction method literature it is sometimes therefore also referred to as the iterative method of Fraser and Roussel, see e.g., \cite{kap3}. 

Moreover, it seems that it was unknown to MacKay that Neishtadt in \cite[Lemma 1]{nei87} had already proven exponential estimates of the form $\kO(e^{-c/\epsilon})$. Neishtadt did, however, not address equilibria and the method he used is (slightly) different from MacKay's.
 Notably it does not contain all equilibria near the slow manifold. Moreover, Neishtadt considered finite dimensional systems and did not consider the Hamiltonian case. 

The method of MacKay can be sketched as follows:
Setting $\dot z=0$ 
in \eqref{e.xys0}  gives, by applying the implicit function theorem
(bearing in mind that ${F}$ is quadratic in $z$), a solution $z=\zeta(w)$ close to 
$-A(w)^{-1}{r}(w)$ provided that $r$ is sufficiently small. The graph $z=\zeta(w)$ will be the improved slow manifold (see Lemma \ref{l.ZetaG} below for details). To show 
that this is indeed an improved slow manifold, one  straightens out the new slow manifold by 
introducing $z_1$ through $z=z_1+\zeta$. Then the equations become
\begin{align*}
 \dot w= \epsilon W_1(w,z_1),\quad \dot z_1 &= Z_1(w,z_1)=
r_1(w) + A_1(w) z_1 + F_1(w,z_1),
\end{align*}
with 
\begin{align}
 r_1(w) = -\epsilon \partial_w \zeta (w) W(w,\zeta(w)).\label{e.rho11}
\end{align}
 Here $\partial_w$ is used to denote the (Frechet) partial derivatives 
$\frac{\partial }{\partial w}$, and we will continue to use this symbol regardless of what object is being differentiated. 
Since $\zeta = \kO(\Vert r\Vert)$ the error vector field has been diminished by a factor $\kO(\epsilon)$. Hence
 $\kM_1 = \{z_1=0\}$ is an improved slow manifold. Note that $\kM_1$ includes nearby equilibria, cf.~\eqref{e.rho11}. Neishtadt in \cite[Lemma 1]{nei87} based his transformations on $z=z_1+\tilde \zeta$, $\tilde \zeta(w) = -A(w)^{-1}r(w)$, giving rise to the error  
\begin{align*}
\tilde r_1 = -\epsilon \partial_w \tilde \zeta (w) W(w,\tilde{\zeta}(w))+F(w,\tilde \zeta(w)).
\end{align*}
Note that as opposed to \eqref{e.rho11}, $\tilde r_1$ does not vanish at all nearby equilibria as the error now comes from two separate contributions.
 
MacKay conjectured that applying this procedure of constrained equilibria, setting $z_i=z_{i+1}+\zeta_i(w)$ with $z=\zeta_i(w)$ solving $Z_i(w,\zeta_i)=0$, successively would lead to a slow manifold  with an error-field of $\kO(e^{-c/\epsilon})$. Note that $w$ is never transformed and that only the inverse of the possibly unbounded operator $A_i$ occurs. Therefore unbounded $A(w)$ can also be accounted for. We prove this result in Appendix \ref{a.genves}. Even though this result has to be attributed to Neishtadt, in contrast to Neishtadt's original result, we  (i) follow  MacKay and allow for an unbounded fast vector field where the $\dot z$ equation is a semilinear evolution equation, and (ii) using MacKay's method we obtain a slow manifold  that contains all nearby
equilibria.


\subsection{Symplectic slow manifolds}
 In this paper we  focus on slow-fast Hamiltonian systems of the form (see also \cite{geller1}):
\begin{align*}
 H &= H(w,z),\quad w=(u,v) \,\,\text{slow},\quad z=(x,y) \,\,\text{fast},\quad
 J=\text{diag}\,(\epsilon J_{\kW},\,J_{\kZ}),
\end{align*}
where, as opposed to regular perturbation theory, the main contribution to the perturbation comes from the symplectic structure operator $J$.  This gives rise to the following equations of motions
\begin{align*}
 \dot w =\epsilon J_{\kW} \nabla_w H,\quad
 \dot z=J_{\kZ} \nabla_z H.
\end{align*}
We assume that the $\dot z$ equation is a semilinear Hamiltonian evolution equation, as detailed later.
 Normally elliptic slow manifolds are  of  particular interest in Hamiltonian systems as stability here is associated with oscillatory normal behavior. In Hamiltonian systems there are also generically invariant manifolds that are not normally hyperbolic, for example families of linearly stable periodic orbits parametrized by energy.
 Here we are interested in  approximately invariant symplectic slow manifolds
 on which we can define a ``slow'' Hamiltonian system.

The Hamiltonian case has previously been considered by Gelfreich and Lerman in \cite{geller1}. They restricted to one fast degree of freedom and also obtained an exponentially accurate slow manifold using an averaging method, similar to the one used by Neishtadt in \cite{nei2}.
 
For {the} Hamiltonian example
\begin{align}
 H=\frac12 x^2+\frac12 y^2+v + \epsilon y f(u),\label{e.nex}
\end{align}
with analytic $f(u) = \sum_{n=1}^\infty \e^{-n}\sin (n u)$ and $\omega =\d x\wedge \d y+\epsilon^{-1} \d u\wedge \d v$, Neishtadt
 showed 
that the slow manifold cannot be improved beyond such an estimate, see \cite{geller1}.  The exponential 
estimate is therefore the best one can aim for in a general non-hyperbolic setting. 
The method of averaging used in \cite{geller1} does not extend to several fast variables primarily due to 
the general lack of control of resonances between the fast variables. On the other hand, it also aims at more than what MacKay and Neishtadt did in \cite{mac1} and  \cite{nei2}: 
the results of \cite{geller1} do not only provide exponential 
estimates of a slow manifold, they also provide an $\kO(1)$-foliation, 
parametrized by the action variable, of almost invariant slow manifolds. 
The method therefore also addresses stability, 
{not only existence of an accurate slow manifold.} 
The reference \cite{mat1} extends the results of \cite{geller1} to infinite dimensional slow dynamics. 
The results {of  \cite{mat1}} 
hold true for \textit{spatially} Gevrey smooth solutions, which allow for a Galerkin approximation that
 separates the vector-field into a bounded one and an exponential small remainder. 
The references \cite{van1, wir1} also provide exponential estimates of 
    particular slow manifolds in geophysical models by obtaining optimal truncations of the 
    ``super-balance equation'' (invariance equation) of Lorenz \cite{lor1}. 
 

\subsection{Improving symplectic slow manifolds}

MacKay, still in \cite{mac1}, suggested a separate method for improving slow manifolds in Hamiltonian systems. The proposed method was described as follows: Consider a Hamiltonian $H=H(p)$ with symplectic form $\omega$ and a slow manifold $\kM_0$.
 Do the following:
\begin{itemize}
 \item Compute an orthogonal symplectic foliation $\kF_p$ so that for every $p\in \kM_0$  
\begin{align*}
\omega({p}_1,{p}_2) = 0,\,{p}_1 \in \kT_p \kF_p,\,{p}_2\in \kT_p {\kM}_0.
\end{align*}
Here $\kT_p \kM_0$ denotes the tangent space of $\kM_0$ at $p$.
\item Let $H_p = H\vert_{\kF_p}$ and solve this for a nearby critical point $p_1=p_1(p)$.
\item Put $\kM_1=\{p_1(p)\}$ and $\omega_{\kM_1} = \omega\vert_{\kM_1}$.
\end{itemize}
Then $(H\vert_{\kM_1},\omega_{\kM_1})$ is an improved slow system. For the further details see \cite{mac1}. However, we believe that this method has some drawbacks. First of all, the method requires the computation of a new slow symplectic form at each step. In fact, we believe that the reason for suggesting an alternative {to the general approach} in the first place, is that one wishes to introduce transformations that preserve the symplectic structure.
Moreover, MacKay's method also requires the  computation of  orthogonal symplectic foliations at each step.

  We will therefore suggest an alternative method that circumvents these issues.
 Our method is then a symplectic extension of MacKay's general approach outlined above. We will at each step straighten out the improved manifold given as the solution $z=\zeta(w)$ of
  \begin{align*}
   J_{\kZ}^{-1} Z(w,z) = \nabla_z H(w,z) = 0,
  \end{align*}
ensuring that the transformation involved in this procedure is symplectic.
The slow symplectic form with symplectic structure matrix $\epsilon^{-1} J_{\kW}^{-1}$ will therefore remain constant throughout the iteration. For analytic Hamiltonian systems where the fast system is a semilinear evolution equation we obtain a symplectic slow manifold with exponentially small error field containing an initially nearby equilibrium, as was also conjectured by MacKay. 
  This is the  main result of this paper. We will present this formally in Theorem \ref{t.main} which we prove in Section \ref{s.hamfin}. In Section \ref{s.Main} we also state and prove a stability result, see Corollary \ref{c.stab}. As opposed to the general case in Appendix \ref{a.genves}, the symplectic nature of the problem requires us to transform the slow variables.  
Note that Lu \cite{Lu} modified our approach,  presented in  a previous preprint version of this paper, and applied it to a more specific problem where the dynamics on the fastest scale is linear, 
see Remark \ref{r.Lu} for more details. 
 In Section \ref{s.psm} we present a dynamical consequence of this result on the persistence of a one degree of freedom slow manifold with exponentially small gaps. 



\section{Exponential estimates for symplectic slow manifolds}
In this section we first introduce some  notation (Section \ref{ss.notation}).  Then, in Section \ref{ss.Setting}, we introduce the setting we work in, in particular our assumptions on the fast Hamiltonian semilinear evolution equation.  In Section \ref{s.Main} we present our main result and in Section \ref{ss.Examples} we consider two examples
where the fast dynamics is a  nonlinear Schr{\"o}dinger equation and a  semilinear wave equation respectively.


\subsection{Some notations}\label{ss.notation}
Let $(\kW,\Vert \cdot \Vert_{\kW})$ and
$(\kZ,\Vert \cdot \Vert_{\kZ})$ be real Banach spaces 
and  let $\kW^{\C}=\kW\oplus \i\kW$ and 
$\kZ^{\C}=\kZ\oplus \i\kZ$, respectively, be their 
complexifications with norms 
$\Vert w_1+\i w_2\Vert_{\kW^{\C}}=
(\Vert w_1\Vert_{\kW}^2+\Vert w_2\Vert_{\kW}^2)^{1/2}$ 
and $\Vert z_1+\i z_2\Vert_{\kZ^{\C}}
=(\Vert z_1\Vert_{\kZ}^2+\Vert z_2\Vert_{\kZ}^2)^{1/2}$.
  Then $f:\kV^{\C}\rightarrow \kZ^{\C}$, with $\kV^{\C}$ an open subset of $\kW^{\C}$, is analytic if it is continuously differentiable, i.e., if there exists a continuous derivative $\partial_w f:\kV^{\C}\rightarrow \kE(\kW^\C;\kZ^\C)$,  where $\kE(\kW^\C;\kZ^\C)$ is
 the Banach space of { bounded} complex linear operators from $\kW^\C$ to $\kZ^\C$ equipped with the operator norm, 
 satisfying the following condition
\[\Vert f(w+h)-f(w)-\partial_w f(w)(h)\Vert = \kO(\Vert h\Vert^2).\]
For later purpose we define $\kE(\kZ):= \kE(\kZ;\kZ)$.
By a real analytic  function we  mean an analytic function which is real valued when its arguments are real. 
The higher order derivatives can be defined inductively and $\partial_w^n f$ becomes a map
\begin{align*}
 \partial_w^n f:\kV^{\C}\rightarrow \kE^n(\kW^{\C};\kZ^{\C}),
\end{align*}
from $\kV^{\C}$ into the Banach space $\kE^n(\kW^{\C};\kZ^{\C})$ of all bounded, $n$-linear maps from $\kW^{\C}\times \cdots \times \kW^{\C}$ ($n$ times) into $\kZ^{\C}$.
See, for example, \cite[Appendix A]{pos1} for a reference on analytic function theory in Banach spaces. When $\kV$ is an open subset of $\kW$  and $\nu>0$ then, as in \cite{geller1}, we define $\kV+\i\nu$ to be the open complex $\nu$-neighborhood of $\kV$:
\begin{align*}
 \kV+\i\nu = \{w\in \kW^{\C}\mid \text{dist}_{\kW^{\C}}(w,\kV) < \nu\},
\end{align*}
where $\text{dist}_{\kW^{\C}}$ is the metric induced from the Banach norm $\Vert\cdot \Vert_{\kW^{\C}}$.
In the following let $\kB_r^{\kZ^\C}(z)= \{ u \in \kZ^{\C}, \|u-z\|<r\}$
denote a $\kZ^{\C}$-open ball of radius $r>0$ around $z$ in the Banach space $\kZ^\C$.


\subsection{Setting and assumptions}\label{ss.Setting}

We consider a slow-fast Hamiltonian system with possibly infinitely many fast degrees of freedom of the form
\begin{equation}
\label{e.HamPDE}
\dot w_0 = \epsilon J_{\kW} \nabla_{w_0} H_0(w_0,z_0), \quad
\dot z_0 = J_{\kZ} \nabla_{z_0} H_0(w_0,z_0) = 
J_\kZ L z_0 +B_0(w_0,z_0).
\end{equation}
Here $w_0=(u_0,v_0) {\in \kW}$,
{$\dim \kW = 2d_\kW$}, are the slow variables 
and $z=(x,y)
{\in\kZ}$  are the fast 
variables,   $\kZ$ is a real Hilbert space with inner product $\langle \cdot, \cdot \rangle = \langle \cdot, \cdot \rangle_\kZ$, $\dim\kZ = 2d_\kZ$
and we allow for $d_\kZ=\infty$.  
We use a subscript on $H_0$ and $B_0$ because we will transform this
system into  a more desirable form in the main theorem (Theorem \ref{t.main}).
The symplectic structure operator is  
\begin{align}
J = \text{diag}\,(\epsilon J_{\kW},J_{\kZ}),\label{eq.JFull}
\end{align}
where $J_\kW$ is the standard symplectic matrix on $\kW= \R^{2d_\kW}$, i.e. 
\begin{align*}
 J_{\kW} &= \begin{pmatrix}
        0 & \id\\
        -\id & 0
       \end{pmatrix}\in \R^{2d_{\kW}\times 2d_{\kW}}
\end{align*}
with $\id$ being
 the identity on $ \R^{d_\kW}$.
We assume that  
\begin{itemize}
\item[(H0)] 
 $ J_{\kZ}$ and $J_\kZ L$ 
are  densely defined, closed, skew-symmetric invertible linear operators on $\kZ$ and $L$ is a bounded, self-adjoint operator.
\end{itemize}
Note that (H0) implies that $J_\kZ$ and $L$ commute.
 
   In the following let $\kV \subset \kW$ be open, let $\kS \subseteq \kZ$
be an open neighbourhood of $0$ and let $w \in \kV + \i\nu_0$, $z \in \kS + \i \sigma_0$
where $\nu_0, \sigma_0>0$.  Assume that $B_0:(\kV+\i \nu_0)\times (\kS+ \i\sigma_0)\to \kZ^\C$ is analytic, as detailed in (H1) below. 
In this setting, 
by semigroup theory
\cite{P83}, the flow of \eqref{e.HamPDE} is well-defined.
The  Hamiltonian of \eqref{e.HamPDE} is 
\begin{equation}
\label{e.HamInf}
H_0(z_0,w_0) = \frac{1}{2}\langle L z_0, z_0\rangle + V_0(z_0,w_0), \quad\mbox{where}\quad
 B_0= J_\kZ \nabla_{z_0} V_0.
\end{equation}
Then $H_0:(\kV+\i \nu_0)\times (\kS+ \i\sigma_0)\to\C$ is analytic { (here and
below we extend the inner product on $\kZ$ analytically
 to the bilinear form $\langle z_1, z_2 \rangle = \langle  z_1, \bar z_2\rangle_{\kZ^\C}$ on $\kZ^\C$, where $ \langle \cdot, \cdot \rangle_{\kZ^\C}$ is the inner product on $\kZ^\C$).}
We also allow $H_0$ to depend continuously on $\epsilon$, but this dependence will be suppressed in the notation.

{ Note that $J$ in \eqref{eq.JFull} defines a symplectic form $ \omega_{\kW \times\kZ}$ on $ \kW\times \kZ$
given by $ \omega_{\kW \times\kZ} = \omega_{\kZ} + \omega_\kW$ where
$ \omega_\kW(w_1, w_2) = (\langle u_1, v_2\rangle - \langle u_2, v_1\rangle)/\epsilon$ for $w_i=(u_i, v_i)$, $i=1,2$, and
$\omega_\kZ(z_1, z_2) := \langle  J_{\kZ}^{-1} z_1, z_2 \rangle_\kZ$ is a bounded, antisymmetric nondegenerate bilinear form on $\kZ$. Moreover 
\eqref{e.HamPDE} is a Hamiltonian system, defined as 
$\omega(X_H, Y) = \partial_Y H(Y_0) Y $  for each $Y = (w,z)$, where $\omega = \omega_{\kW \times\kZ}$ and $X_H(Y_0)= \dot Y_0$ is the right hand side of \eqref{e.HamPDE}, (see e.g., Definition 3.3.1 of \cite{AbrMars}). As usual we define
$\nabla_z H(w_0,z_0)$ by the requirement that 
\begin{equation}\label{e.grad}
\partial_z H(w_0, z_0) z = \langle \nabla_z H(w_0,z_0), z\rangle 
~\mbox{for all}~z\in \kZ^\C, (w_0, z_0) \in (\kV+\i\nu_0)\times (\kS+\i\sigma_0).
\end{equation}
}
We define $\kZ_1:= D(J_\kZ)$.  For initial data in $\kZ_1$ the solution of \eqref{e.HamPDE} is differentiable in time in $\kZ$ and so a classical solution in $\kZ$, see \cite{P83}.
The space $\kZ_1$ is a real Hilbert space with inner product 
\[
\langle\langle z_1,z_2 \rangle \rangle=\langle J_\kZ z_1,J_\kZ z_2 \rangle
\]
 and norm $\Wert \cdot \Wert$ which is stronger than $\Vert \cdot \Vert$, i.e., 
  \begin{align}\label{e.ZZ1norm}
  \Vert z\Vert \le  \Wert z\Wert\quad \forall z\in \kZ_1.
  \end{align}
  Here we assume without loss of generality
 \[c_J:=
 \|J_\kZ^{-1} \|_{\kE(\kZ)}  \leq 1.
 \]
 (If $c_J>1$ then we change the inner product on $\kZ$ to $c_J \langle \cdot, \cdot \rangle$
  and $J_\kZ$ to $ c_J J_\kZ$, 
 $\nabla_z$ to $c_J^{-1} \nabla_z$ to achieve $\|J_\kZ^{-1} \|_{\kE(\kZ)}\leq 1$.) 
  Moreover by definition $J_\kZ^{-1}$ maps $\kZ$ onto $\kZ_1$.
  Note that $B_0:\kV\times \kS\to \kZ$ and $  \nabla_z V_0=J_\kZ^{-1}B_0$ 
  imply that 
  \begin{equation}\label{e.nablazV}
  \nabla_z V_0:\kV\times \kS\to \kZ_1.
  \end{equation}
  Taylor-expanding $V_0$ around $z=0$
   then gives
 \begin{align}
   H_0(w_0,z_0) = h_0(w_0)+\langle r_0(w_0),z_0\rangle+\frac12 \langle (L+ a_0(w_0))z_0,z_0\rangle + f_0(w_0,z_0),\label{e.H0Z0}
 \end{align}
 where $f_0= \kO(\|z_0\|^3)$ and  $(w_0,z_0)\in \kW\times \kZ$.
 Then \eqref{e.nablazV} implies that
 \begin{align}
 r_0:&\,\kV \to \kZ_1,\label{e.rZ1}
 \end{align}
 and 
 \begin{align*}
 a_0:&\,\kV \to \kE(\kZ; \kZ_1), \quad
 F_0 = \nabla_z f_0: \,\kV\times \kS \to \kZ_1.\nonumber
 \end{align*}

 We then consider Hamiltonians of the form \eqref{e.H0Z0}  
 and assume the following:
 
\begin{enumerate}\item[(H1)]  The Hamiltonian $H_0$ is real analytic and uniformly bounded on $(w_0,z_0)\in (\kV+\i\nu_0)\times (\kS+\i\sigma_0)$, where $\kV \subset \kW$ is open, $\kS \subset \kZ$ is an open neighbourhood of $0$ in $\kZ$ and $\sigma_0, \nu_0>0$. 
 \item[(H2)]  
  The map
\begin{align*}
  a_0: \kV+\i\nu_0 \to \kE(\kZ;\kZ_1)
\end{align*}
is  real analytic. Moreover, 
\[
\Vert a_0\Wert_{\nu_0} :=
\sup_{w_0\in \kV+\i\nu_0} \Vert a_0(w_0)\Vert_{\kE(\kZ; \kZ_1)}
\leq C_{a_0},
\]
and
\[
\Wert A_0(w_0)^{-1}\Wert_{\nu_0} :=\sup_{w_0\in \kV+\i\nu_0} \Vert A_0(w_0)^{-1} \Vert_{\kE(\kZ_1)} \le K_0/2,
\]
where \[A_0(w):= L + a_0(w).\]
\item[(H3)]
The functions 
\begin{align*}
{r}_0&=\nabla_{z} H_0\vert_{z_0=0}:(\kV+\i\nu_0) \rightarrow \kZ_{1}^\C,\\
h_0&=H_0\vert_{z_0=0}:(\kV+\i\nu_0) \rightarrow {\C},
\end{align*}
and  
$f_0$ with $$\nabla f_0 ={F}_0:(\kV+\i\nu_0)\times (\kS+\i\sigma_0)\rightarrow \kZ_{1}^\C,$$
 are real analytic and uniformly bounded:
\begin{align}
 \Wert {r}_0 \Wert_{\nu_0} &\leq  \delta_0  <\infty\nonumber \\
\Vert {h}_0 \Vert_{\nu_0}&\leq  C_{h_0} <\infty, \nonumber \\
\Vert {F}_0 \Wert_{\nu_0,\sigma_0} & \leq C_{F_0}<\infty, \nonumber\\
\Vert {f}_0 \Vert_{\nu_0,\sigma_0} & \leq C_{f_0}<\infty. \nonumber\\
\Vert \partial_w {h}_0 \Vert_{\nu_0} & \leq C_{h_0}' <\infty \nonumber.
\end{align}
 Here we denote  by 
$\Vert \cdot \Vert_{\nu_0,\sigma_0}$ 
 the sup-norm taking over the domain $(\kV+\i\nu_0) \times (\kS+\i\sigma_0)$. Moreover we define   $\Vert \cdot \Wert_{\nu_0,\sigma_0}$ 
to be the the sup-norm taking over the domain $(\kV+\i\nu_0) \times (\kS+\i\sigma_0)\subset \kW^{\C}\times \kZ^{\C}$ in the $\kZ_1^\C$ norm.
\end{enumerate}
Note that $r_0$ is measured in the $\kZ_1$-norm $\Wert \cdot \Wert$ because of \eqref{e.rZ1}.


\subsection{Main result}\label{s.Main}
We are now ready to formulate our main result:

\begin{theorem}\label{t.main}
Assume (H0)-(H3) and  let  $\nu_0>\nu > 0$, $\sigma_0 >\sigma> 0$.
 Then for
 $\delta_0>0$ and  $\epsilon>0$ sufficiently small the following holds true: There exists a symplectic transformation 
 $(w,  z)\mapsto (w_0,z_0)$, $(  w,  z)\in (\kV+\i {\nu})\times (\kS+\i \sigma)$ with 
  $\Wert   z-z_0 \Wert_{ {\nu}, {\sigma}},\,\Vert   w-w_0 \Vert_{ {\nu}, {\sigma}}=\kO(\delta_0)$ transforming \eqref{e.HamInf} into
 \begin{align}
   H(  w,  z)=H_0(w_0,z_0)=  h(  w)+ \langle r(  w), z\rangle
    +\frac12 \langle (L+a( w))z,  z\rangle +  f(  w,  z),\label{e.Htilde}
 \end{align}
with $f= \kO(\|z\|^3)$,  $F=\nabla f$,
\begin{align*}
\Vert h-h_0\Vert_\nu, \Vert a-a_0\Wert_\nu,  \Vert f-f_0\Vert_{\nu,\sigma},
\Vert F-F_0\Wert_{\nu,\sigma} = \kO(\delta_0)
\end{align*}
  and where
\begin{align*}
 \Wert r\Wert_{\nu} \le C_1 e^{-C_2/\epsilon}.
\end{align*}
 Here $C_1$ and  $C_2$  are positive constants that depend solely on 
$C_{a_0}$, $ C_{h_0}'$, $K_0$, $C_{F_0}$, $C_\kS$, $C_H$, $\sigma_0$, $\nu_0$,  
$\sigma$, $ \nu$,
where $C_\kS = \Vert z_0\Vert_{\sigma_0}$. 
\end{theorem}
In other words: $\{  z=0\}$ is an \textnormal{almost} invariant symplectic slow manifold. 
 Note that the flow corresponding to the transformed Hamiltonian $H$ is well-defined
 because the $\dot z$ equation is again a  semilinear evolution equation of the form considered in \cite{P83}.


Next we address the stability of the slow manifold:

\begin{cor}\label{c.stab}
Under the assumptions of Theorem \ref{t.main} consider the transformed Hamiltonian \eqref{e.Htilde}   on the real domain
    $\kV \times \kS$. If $A_0(w) = L+ a_0(w)$ is positive  definite then 
    $$\ell(w,z)=\frac12 \langle z, A(w)z\rangle +{f}(w,z),$$ is an approximate Lyapunov function 
    for $\Vert z\Vert$, $\delta_0$ and $\epsilon$ sufficiently small, and there exist constants 
{$c_1$} and {$c_2$} so that
\begin{align*}
 \Vert z(t)\Vert \le \kO(e^{-c_1/\epsilon})\quad
  \text{for}\quad 0\le t\le c_2\epsilon^{-2},
\end{align*}
when $z(0)=0$, provided that $w(t) \in \kV$ for $ 0\le t\le c_2\epsilon^{-2}$.
 \end{cor}
\bpr
By assumption $A_0(w)$ is positive definite  and hence, 
 since by Theorem \ref{t.main}   we have $\Vert a-a_0\Wert_\nu = \kO(\delta_0)$,
 so is $A(w)$   for $\delta_0$ small. Therefore there exist constants ${\lambda_1}{>0}$ and ${\lambda_2}{>0}$ so that
\begin{align}
  {\lambda_1} \Vert z \Vert^2 \le \ell(w,z) \le {\lambda_2} \Vert z \Vert^2\label{e.lmin}
\end{align}
for $\|z\|$ small and all $w \in \kV +\i\nu$.
Differentiating $\ell(w,z)$ in $t$ we then obtain
\begin{align*}
\dot \ell(w(t), z(t)) &= \partial_z \ell J_{\kZ}(  r+L z+ a z + \nabla_z   f) 
+\epsilon \partial_w \ell  J_{\kW}\nabla_w H \\
& = \langle L z+a(w) z + \nabla_z  f,  J_{\kZ}(   r+L z+a(w) z + \nabla_z f) \rangle 
+\epsilon \partial_w \ell J_{\kW}\nabla_w H \\
&= \langle L z  + a(w) z +\nabla_z   f,  J_{\kZ} r \rangle +\epsilon \partial_w \ell J_\kW\nabla_w  H\\
&\le C_3  e^{-C_2/\epsilon} +C_4\epsilon \sup_{w {\in \kV+\i \nu}} \ell(w,z(t)),
\end{align*}
for some constants $C_3$ and $C_4$
as long as $(w(t),z(t))\in \kV\times \kS$. Here we have used that $\langle z_1, J_{\kZ} z_1\rangle =0$ for all $z_1$ and a Cauchy estimate on $\partial_w \ell$ and $\partial_w H$.  Note that $\dot \ell(w(t), z(t))$ is defined at all $z(t) \in \kS$.
Integrating this inequality from $s=0$ to $t$ and using \eqref{e.lmin} we find that any initial
data $z(0)=0$
\begin{align*}
  {\lambda_1}\Vert z(t)\Vert^2 \le \ell(w(t),z(t)) & \le C_3  t \e^{-C_2/\epsilon} 
 +C_4\epsilon \int_0^t \sup_{w\in \kV+\i\nu} \ell(w,z({s})) \d {s}\\
& \le C_3  t e^{-C_2/\epsilon} +C_4 {\lambda_2}\epsilon \int_0^t \Vert z{(s)}\Vert^2 \d {s}.
\end{align*}
We have here used that $\ell(w(t), z(t))|_{t=0}=0$ since $z(0)=0$ by assumption. Then by Gronwall's inequality in integral form \cite{Bellman} we obtain
\begin{align*}
 \Vert z(t)\Vert^2 \le C_3{\lambda_1}^{-1}t e^{-C_2/\epsilon}e^{C_4   {\lambda_2}  {\lambda_1}^{-1}\epsilon t},
\end{align*}
and therefore while $0\le t\le C_2 {\lambda_1}/(2C_4  {\lambda_2} \epsilon^2)=c_2/\epsilon^2$ we have
\begin{align*}
  \Vert z(t)\Vert\le \sqrt{\frac{C_3 C_2}{2C_4  \lambda_2}} e^{-C_2/(4\epsilon)}/\epsilon = \kO(\e^{-c_1/\epsilon}) ,
\end{align*}
where $c_1< C_2/4$,
completing the proof.
\epr
Note that this upper estimate $\kO(\epsilon^{-2})$ on the time interval is large, even on the slow time scale $\tau= \epsilon t$.


 \subsection{Examples}\label{ss.Examples}
 In this section we present our two main examples.
 We consider PDEs with periodic boundary conditions, i.e., on the circle $\S^1 \simeq \R/(2\pi\Z)$.  We frequently use
 the Hilbert space $\kL^2(\S^1; \C^d)$ of square integrable functions with inner product
  \[
  \langle x_1, x_2 \rangle_{\kL^2(\S^1, \C^d)}
  = \int_0^{2\pi} x_1(s) \cdot \overline{x_2(s)} \d s
  \]
  and the Sobolev spaces $\kH_k(\S^1; \C^d)$ as  the  spaces containing the $\kL^2(\S^1; \C^d)$ functions with $k$ weak derivatives equipped with the $\kH_k$ inner product
  \begin{equation}\label{e.HkProduct}
  \langle x_1,x_2\rangle_{\kH_k(\S^1; \C^d)} = \langle (1-\partial_s^2)^k x_1,x_2\rangle_{\kL^2(\S^1; \C^d)}.
  \end{equation}
  The examples are also used later in Remark \ref{r.genFct} to exemplify an abstract construction related to the introduction of an generating function. 

\subsubsection{Nonlinear Schr{\"o}dinger equation}
\label{ssec:nse}
Consider the nonlinear Schr\"odinger equation
\begin{equation}
  \label{e.nse}
  \i \, \partial_t z
  = \partial^2_{s} z -z+ \partial_{\overline{z}} U (w,z,\overline{z}),
\end{equation}
on the circle $\S^1$
coupled to a slow system
\[
\dot w = \epsilon J_\kW \left(\nabla_w h(w) +\int_0^{2\pi} \tfrac12\nabla_w U(w,z,\overline{z}) \d s\right),
\]
where $h:\kV+\i\nu_0 \to\C$ is real analytic with $\nu_0>0$
and $\kV\subset \kW$ open.
Furthermore, using the real coordinates $(x,y)\in \C$, 
where  we identify $\R^2 \simeq \C$  via $z = x + \i y= (x,y)$,
we assume that   $\sU(w,x,y) = U(w,x+\i y, x- \i y)$ is analytic in $w\in \kV+\i\nu_0$ and  $x=\Re z$ and $y=\Im z$. 

In the coordinates $(x,y)$ the nonlinear Schr\"odinger equation \eqref{e.nse}
takes the form
\begin{equation}\label{e.Rnse}
 \dot x = \partial_s^2 y  -y+\tfrac12 \nabla_y \sU(w,x,y) ,\quad \dot y =  -\partial_s^2 x +x- \tfrac12 \nabla_x \sU(w,x,y).
\end{equation}
This follows from the fact that 
\[
\partial_{\bar z}U(w,z,\bar z) = \tfrac12 \partial_x \sU(w,x,y)  +  \tfrac\i2 \partial_y \sU(w,x,y).
\]
We can rewrite \eqref{e.Rnse} as 
\[
{\dot x \choose \dot y} = J \nabla H(w,x,y).
\]
Here  $\nabla = \nabla^{\kL^2(\S^1;\R^2)}$  is the gradient w.r.t. the $\kL^2(\S^1;\R^2)$
pairing,  
\begin{equation}
\label{e.JL2}
J = J_{\kL^2(\S^1;\R^2)}=\begin{pmatrix}
                            0 & \id_{\kL^2}\\
                            -\id_{\kL^2} &0 
                           \end{pmatrix},
\end{equation}
{with $\id_{\kL^2}$ the identity on $\kL^2 = \kL^2(\S^1;\R)$},
is the standard symplectic structure { operator on $\kL^2 \times \kL^2$} and  
\begin{equation}
  \label{e.HamiltonianNSE}
  H(w,x,y) = \frac{1}{2} \int_{\S^1} 
         \bigl( 
           \lvert \partial_s x \rvert^2 +\vert \partial_s y \rvert^2 +x^2+y^2+ \sU(w,x,y)
         \bigr) \, \d s+h(w).
\end{equation}
Note that 
 the Hamiltonian
 $H$ is well defined on $\kZ := \kH_1(\S^1;\R^2)$, but not on
$\kL^2(\S^1;\R^2)$.  Defining { the symplectic form $\omega_{\kZ}(z_1,z_2)$
on $\kZ$ as}
\begin{align}
  \label{e.J_nls}
\omega_{\kZ}(z_1,z_2) =  \langle J_\kZ^{-1} z_1, z_2 \rangle_\kZ
  & =  \int_0^{2\pi}( x_1(s) y_2(s) - y_1(s) x_2(s)) \d s
   = 
   \omega_{\kL^2\times \kL^2}( z_1, z_2),
\end{align}
{ where
\begin{equation}\label{e.omegaL2}
 \omega_{\kL^2\times \kL^2}( z_1, z_2)= \langle J^{-1}  z_1, z_2 \rangle_{\kL^2(\S^1;\R^2)}
\end{equation}
is the standard symplectic form on $\kL^2\times \kL^2$}, 
we see { from \eqref{e.HkProduct}} that  $J_\kZ = (1-\partial^2_x) J$ so that  $J_\kZ^{-1} \colon
\kH_2(\S^1;\R^2)\to \kL^2(\S^1;\R^2)$.   The Laplacian is diagonal
in the Fourier representation with eigenvalues
$-k^2$.  Hence, $ \spec\, J_\kZ = \{ \pm \i( k^2+1) \colon k \in \Z\}$ so that $J_\kZ$
generates a unitary group on $\kL^2(\S^1;\R^2)$ and on $\kZ
=\kH_1(\S^1,\R^2)$. This  shows that  
the Hamiltonian of the nonlinear Schr\"odinger equation \eqref{e.nse}
takes the form \eqref{e.HamInf} with  
  \[
 L = \id, \quad   V(w,x,y) =\frac12 \int_0^{2\pi} \sU(w,x(s), y(s))\d s  
\]
where $z \in \kZ$. The fact that $H$ and $B= J_\kZ\nabla V$
are analytic as maps from $(w,z) \in (\kV+\i \nu)\times(\kS+\i\sigma)$
to $\C$ and $\kZ$ respectively can be deduced from the fact 
 that $\kH_1(\S^1;\R)$ is an algebra, see   \cite[Theorem 5.23]{Adams}.
Moreover 
\[
a_0(w) =\tfrac12 (1-\partial_s^2)^{-1} \partial_z^2 \sU(w,0).
\]
If $\partial_z^2 \sU(w,0)$
is small then $A(w) = \id + a_0(w)$
is invertible on $\kZ_1= \kH_3(\S_1;\R^2)$ as required in (H2).  
Under this assumption conditions (H0-H3) are satisfied.
By choosing $\partial_z \sU(w,0)$  sufficiently small  we can  make
$\delta_0$ sufficiently small as required  in Theorem \ref{t.main}.

\subsubsection{Semilinear wave equation}
\label{ssec:swe}

Consider a semilinear wave equation of the form
\begin{subequations}
\label{e.swe}
\begin{equation}\label{e.swe1}
\dot x = y,\quad
 \dot y =(\partial_s^2-1) x - \nabla_x U (u,x) 
\end{equation}
with $z=(x,y)$ and where $x=x(t,s)$, $y=y(t,s)$ with $s\in \S^1$, for simplicity. This system is coupled to a set of slow ordinary differential equations for the evolution of $w=(u,v)=(u,v)(t)\in \kW =  \R^{2\d_{\kW}}$ of the form
\begin{equation}\label{e.swe2}
\dot u  =\epsilon v,\quad
 \dot v = -\epsilon \left(\nabla_u h(u) + \int_0^{2\pi} \nabla_u U(u,x)\d s\right).
\end{equation}
\end{subequations}
{ We consider \eqref{e.swe} on $\kW \times \kZ$ where}
$\kZ=\kH_1(\S^1;\R)\times \kL^2(\S^1;\R)$, 
with inner product 
\begin{align*}
 \langle z_1,z_2\rangle  &  = \langle x_1,x_2\rangle_{\kH_1(\S^1;\R)} + \langle y_1,y_2\rangle_{\kL^2(\S^1;\R)},\quad
  z_1=\begin{pmatrix} x_1\\y_1\end{pmatrix},\,z_2=\begin{pmatrix} x_2\\y_2\end{pmatrix}.
\end{align*} 
  This system is Hamiltonian with
 Hamiltonian and symplectic structure operator given by 
\begin{align*}
 H(w,z) &= \frac12 \vert v\vert^2+\frac12 \langle z,z\rangle  + h(u)+\int_0^{2\pi} U (u,x)\d s,\\
 J_\kZ z &= {
        y\choose
        \partial_s^2x- x
       },
\end{align*}
and  \[
L = \id,\quad
 V(u,x) = \int_0^{2\pi} U(u,x) \d x, \quad 
B(u,x) = { 0 \choose  - \partial_x U(x,y)}.
\]

{ Note that $J_\kZ$ defines the symplectic form 
$\omega_\kZ(z_1, z_2) = \langle  J_{\kZ}^{-1}  z_1,z_2\rangle_\kZ $
on $\kZ$ which (as in the case of the nonlinear Schr\"odinger equation, cf. \eqref{e.J_nls}) satisfies
\begin{equation}\label{e.omegaSWE}
 \omega_\kZ(z_1,z_2) = \omega_{\kL^2 \times \kL^2}(z_1,z_2),\quad (z_1,z_2)\in \kZ\times \kZ.
\end{equation}
Here $\omega_{\kL^2 \times \kL^2}$ is the standard symplectic form on $\kL^2 \times \kL^2$, see \eqref{e.omegaL2} and \eqref{e.JL2}. To prove \eqref{e.omegaSWE} note that
 for $z_1, z_2 \in \kZ$
with $z_i = (x_i, y_i)$ and $x_i \in \kH_1$, $y_i \in \kL^2$, $i=1,2$, we have
\begin{align*}
\langle  J_{\kZ}^{-1} z_1, z_2\rangle_\kZ &= \langle {(\partial_s^2-1)^{-1} y_1 \choose x_1}  , {x_2\choose y_2}\rangle_{\kZ} =
 \langle x_2, (\partial_s^2-1)^{-1} y_1\rangle_{\kH_1}  + \langle x_1,  y_2\rangle_{\kL^2}\\
 &  = 
 \langle (1-\partial_s^2)  (\partial_s^2-1)^{-1}  x_2, y_1\rangle_{\kL^2}  +  \langle x_1,  y_2\rangle_{\kL^2}\\
& = 
 \langle x_1, y_2\rangle_{\kL^2} -\langle x_2, y_1\rangle_{\kL^2}
  = \omega_{\kL^2 \times \kL^2}(z_1, z_2).
\end{align*}
Here we used that $(\partial_s^2-1)^{-1}$ is self-adjoint on $\kH_1$ and the definition \eqref{e.HkProduct} of the inner product on $\kH_1$.}

We assume that $U:\R^{\d_{\kW}+1} \rightarrow \R$ is an analytic function. Then $V$ and $B$ are also analytic as functions from $\R^{d_{\kW}}\times \kH_1(\S^1;\R)$ to $\R$ and $\kZ$ respectively, which follows  from the theory of superposition operators, see   \cite[Theorem 5.23]{Adams},
and so the system \eqref{e.swe} is well-posed on  $\kW\times \kZ$ \cite{P83}.
Finally
\begin{align*}
 D(J_\kZ) =\kH_2(\S^1;\R)\times \kH_1(\S^1;\R).
\end{align*}
Then we note that due to  the definition  { \eqref{e.grad}}  of $\nabla_z H$ we have for all $\tilde z \in \kZ$ that
\[
 \partial_z H(w, z)\tilde z = \langle \nabla_z H(w, z),\tilde z \rangle 
\]
and hence, due to 
\[
\langle \nabla_x U(u,x), \tilde x \rangle_{\kL^2(\S^1;\R)} = \langle (1-\partial_x^2)^{-1}\ \nabla_x U(u,x), \tilde x \rangle_{\kH_1(\S^1;\R) }
\]
we get
\begin{align*}
 \nabla_z H
  = {x+  (1-\partial_x^2)^{-1}
  \nabla_x U
 \choose        y }.
\end{align*}
The assumptions on analyticity (H0-H3) are  satisfied due to the analyticity assumption on $U$
 provided that  $\partial_x^2U(u,0)$ is sufficiently small. Note that 
the smoothing property in (H2), (H3) is  a consequence of the appearance of the isomorphism $(1-\partial_s^2)^{-1}:\kL^2(\S^1;\R)\rightarrow \kH_2(\S^1;\R)$ in the expression above. Moreover, $A_0(w_0)^{-1}$  exists and is bounded as an operator from $\kZ_1$ into $\kZ_1$ as required in (H2) if $\partial_x^2 U(u,0)$ is small so that $A_0(w_0) = \id + a_0(u_0)$ is a small $\kZ_1$-perturbation of the identity. By choosing $\partial_x U(u, 0)$ sufficiently small we can make $\delta_0$ small as required in Theorem \ref{t.main}.


\section{Proof of  main result}\label{s.hamfin}

We start with some preliminary lemmas which will be needed in the proof.
Then, in Section \ref{ss.GenFct} we introduce the generating functions for the symplectic transformations we consider. 
In Section \ref{ss.symplTrafo} we prove that those symplectic transformations are well-posed. We then set up an iterative lemma (Section \ref{ss.ItLemma}) which we use to prove the main theorem in Section \ref{ss.ProofMain}.

\subsection{Preliminary lemmas}
We   frequently need the following Cauchy estimate \cite{pos1}:

\begin{lemma}\label{l.cest}
Assume that $\kW^\C$ and $\kZ^\C$ are Banach spaces, let $\kV^\C \subseteq  \kW^\C$ be open and assume that $f:\kV^\C\rightarrow \kZ^\C$ is analytic and that $f$ is bounded on $\kB_\nu(w_0)\subset \kV^\C$
 {for some $\nu>0$}. Then for $n \in \N$
\begin{align}
 \Vert \partial^n_w f(w_0)\Vert \le n! \frac{\sup_{w\in \kB_\nu(w_0) }\Vert f(w)\Vert}{\nu^n}.\label{e.cest}
\end{align}
\end{lemma}

\begin{remark}
Let $f:\kV+\i\nu\rightarrow \kZ^\C$ be analytic and bounded
and let $\nu>\xi>0$. Then we can apply the estimate \eqref{e.cest} to any $w_0\in \kV+\i(\nu-\xi)$ to obtain:
\begin{align*}
\sup_{w_0\in \kV+ \i(\nu-\xi)}\Vert \partial^n_w f(w_0)\Vert \le n! \frac{\sup_{w\in \kV+\i\nu}\Vert f(w)\Vert}{\xi^n},
\end{align*}
which we will write compactly as
\begin{align*}
\Vert \partial^n_w f\Vert_{\nu-\xi} \le n!\frac{\Vert f\Vert_{\nu}}{\xi^n}.
\end{align*}
This is the form of Cauchy's estimate that we will be using. 
\end{remark}

 We will also use the following generalized version of Taylor's theorem \cite{pos1}:
\begin{lemma}
 If $f:\kV^\C\rightarrow \kZ^\C$ is $n$ times continuously differentiable, $n\ge 1$, and if the segment $w+sh$, 
 $0\le s\le 1$, is contained in $\kV^\C$, then 
\begin{align*}
 f(w+h) &= f(w)+\partial_w f(w)(h)+\cdots +\frac{1}{(n-1)!}\partial_w^{n-1} f(w)h^{n-1}\nonumber\\
&+\int_0^1 \frac{(1-s)^{n-1}}{(n-1)!}\partial_w^n f(w+sh)(h,\cdots,h) \d s.
\end{align*}
The integral remainder is bounded by $\frac{\Vert h\Vert^n}{n!}\sup_{0\le s \le 1} \Vert \partial_w^n f(w+sh)\Vert$.
\end{lemma}
Here, we write for $m \in \N$,
\begin{align*}
 \partial_w^m f(w)h^m = \partial_w^m f(w)(h,\ldots,h).
\end{align*}

To continue we introduce the following notation:
Under assumption (H0-H3) for $R>0$ such that $\kB_R^{\kZ^\C}(0) \subseteq \kS+\i\sigma_0$, define
 \begin{equation}\label{e.CF''}
  C''_{F_0}[\nu_0, R] := 
\sup_{\substack{ w_0\in \kV+\i\nu_0, \\ \|z_0\| \leq R}} \Vert \partial^2_z F_0(w_0, z_0)\Vert_{\kE(\kZ\times\kZ; \kZ_1)}.
 \end{equation}
 
We  need the following lemma, which deals with solutions of the equation $\dot z=0$,
 for the construction of improved slow manifolds:
 
\begin{lemma}\label{l.analimpl}
Assume (H0-H3) with the subscripts dropped   and assume that 
\begin{equation}
\delta <  2/(K^2 C''_{F}[\nu, K \delta])),\label{e.deltaCond}
\end{equation}
 and that $K \delta <   \sigma$.
Then  
\begin{align}
 0=r(w) + A(w)z+{F}(w,z),\label{e.xyeqn}
\end{align}
 has a locally unique solution $z=\zeta(w)\in \kZ_1^\C$ satisfying:
\begin{align*}
 \Wert \zeta(w)\Wert \le K \Wert r(w) \Wert,
\end{align*}
for every $w\in \kV+\i\nu$. Moreover $\zeta(w)$ is analytic in $w\in \kV+\i\nu$:
 \begin{align*}
   \zeta&\in C^\omega(\kV+\i\nu;\kZ^\C_{1}).
  \end{align*}  
\end{lemma}

\bpr
  This is  just a consequence of the uniform contraction theorem, see e.g., \cite[Section 1.2.6]{Henry}, where  \eqref{e.deltaCond} is the condition to
ensure the contraction property. We include the proof to verify the estimates. 
Re-arranging \eqref{e.xyeqn}
and applying the inverse $A(w)^{-1}$ gives
\begin{align}
 \zeta(w)= -A(w)^{-1}(r(w)+{F}(w,\zeta )).\label{e.eta}
\end{align}
Put $\tilde\zeta(w)=-A(w)^{-1} r(w)$ and write $\zeta=\tilde\zeta(w)+z$
so that
\begin{align*}
   z = \Pi(w,z):= -A(w)^{-1}F(w,\tilde\zeta(w)+  z)=-(L+ a(w))^{-1}F(w,\tilde\zeta(w)+  z)\label{e.mueqn}
\end{align*}
for $w \in \kV+\i\nu$.
Note that $\Wert \tilde\zeta(w)\Wert\le \frac{K}{2}\Wert r(w)\Wert$ by (H2). 
We will denote this upper bound by $\rho(w)=\frac{K}{2} \Wert r(w)\Wert$ and highlight that $\rho(w) \le \frac{K}{2}\delta$. 
 We will show that $\Pi(w,\cdot)$ is a 
contraction on $\kB^{\kZ_1^\C}_{\rho(w)}(0)$ for each $w \in \kV+\i\nu$.
  Note that $\kB^{\kZ_1^\C}_{\rho(w)}(\tilde\zeta(w))\subset \kS+\i\sigma$ because $\sigma> K \delta> 2\rho(w)$
   by assumption. By Taylor's formula we have for $\|z\|\leq K\delta $ that
\[
\Wert{F}(w,z)\Wert=\Vert\int_0^1(1-s)\partial_z^2 {F}(w,sz)z^2 \d s\Wert
\le \frac{1}{2} C_F''[\nu, K\delta] \|z\|^2.
\] 
Therefore for $z\in \kB^{\kZ_{1}^\C}_{\rho}(0)$ and $w \in \kV+\i\nu$
\begin{align*}
 \Wert \Pi(w,z)\Wert 
 \le \frac{K}{4}  C_F''[\nu, K\delta] \Vert \tilde\zeta(w)+z\Vert^2 
\le \frac{K}{4} C_F''[\nu, K\delta] (2\rho)^2
\le K^2C_F''[\nu, K\delta] \delta \rho/2 <\rho
\end{align*}
using that $2\rho \le K\delta$ and  assumption \eqref{e.deltaCond}, 
and hence $\Pi(w,\cdot):\kB^{\kZ_{1}^\C}_\rho(0) \rightarrow \kB^{\kZ_{1}^\C}_\rho(0)$.
 Next,  by Taylor's formula 
 $$\partial_z {F}(w,z)=\int_0^1\partial_z^2 {F}(w,tz)z \d t,$$
from which we obtain that for $z\in \kB^{\kZ_{1}^\C}_{\rho}(0)$
  \[
  \Wert \partial_z {F}(w,\tilde\zeta(w)+z)\Wert \le C_F''[\nu,K\delta] \Vert \tilde\zeta(w)+z\Vert
  \leq C_F''[\nu,K\delta] K \delta.
  \]
  Therefore for $z\in \kB^{\kZ_1^\C}_\rho(0)$
\begin{align*}
 \Wert \partial_z \Pi(w,z)\Wert 
&\le{K^2} C_F''[\nu, K\delta]\delta/2<1,
\end{align*}
using \eqref{e.deltaCond}. This shows that $\Pi(w,\cdot)$ is a contraction on the ball $\kB^{\kZ^\C_{1}}_\rho(0)$ and therefore there exists
 a unique fixed point $  z(w)$ of $\Pi(w,\cdot)$. In particular, $\zeta(w)=\tilde \zeta(w)+ z(w)$ solves \eqref{e.eta}
  and $\Wert \zeta(w)\Wert \le 2\rho=K\Wert r(w)\Wert$. By \cite[Section 1.2.6]{Henry} 
  the map $\zeta: \kV+\i\nu\to \kZ^\C_{1}$
  is analytic.
\epr

 
\subsection{Generating functions}\label{ss.GenFct}
 The proof of Theorem \ref{t.main} is based on  successive symplectic transformations.
We will in the following consider   the Hamiltonian  $H(w,z)$
in place of $H_0(w_0, z_0)$ from \eqref{e.HamInf},  satisfying the assumptions (H0-H3), with subscripts removed,
on $(w,z)\in (\kV+\i\nu)\times (\kS+\i\sigma)$.
 As for the general case considered in Appendix \ref{a.genves} and explained
 in the introduction we improve the manifold $\kM_0 = \{z=0\}$  by solving $\dot z =0$
for $z = \zeta(w)$ using Lemma \ref{l.analimpl}. 
 We generate a symplectic transformation $\Psi$ from the non-canonical transformation $z=\zeta(w)+z_+$ through  a generating function introduced in the following lemma.

In this section we formally define $\Psi$ and show that it is symplectic. In the next section we will
show that $\Psi$ is well-defined.

\begin{lemma}\label{l.GenFct}
Assume (H0-H3) with subscript dropped.
Then there are projectors $\P_x$, $\P_y$  on $\kZ$ such that 
\[
\P_x+\P_y = \id,
\quad \P_x \P_y  =0,\quad \P_x \kZ \bot \P_y \kZ,
\]
 with the following property: 
Let  $z = (x,y)$ with $x = \P_x z$, $y=\P_y z$,  $w=(u,v)$ and let
\[
g(u,v_+,x,y_+)  =- \langle J_{\kZ}^{-1}\zeta(u,v_+),(x,y_+) \rangle_\kZ.
\] 
Then 
 $(u_+,v_+,x_+,y_+)\mapsto \Psi(u_+,v_+,x_+,y_+) = (u,v,x,y) $ formally defines 
{a symplectic transformation} given implicitly by the equations:
\begin{align}
  x_+ = x-\zeta^x(u,v_+),&\quad y =y_++\zeta^y(u,v_+),\notag\\
u_+ = u+\epsilon \nabla_{v_+} g(u,v_+,x,y_+),&
\quad v=  v_++\epsilon \nabla_u g(u,v_+,x,y_+).
\label{e.G-psi}
\end{align}
\end{lemma}

 \bpr
 To construct the generating function  we first define a Hilbert space $\tilde \kZ = \tilde\kX\times \tilde\kX$ such that \eqref{e.HamPDE}
 is Hamiltonian on  $\tilde \kZ$ with respect to the standard symplectic structure matrix
 \begin{equation}\label{e.JTildeZ}
J_{\widetilde\kZ}= \begin{pmatrix}
  0 & \id_{\tilde\kX}\\
  -\id_{\tilde\kX} & 0
  \end{pmatrix}.
\end{equation}
We claim that
 $\tilde \kZ = \kZ_{-1/2}$ is the dual space of $\kZ_{1/2} = D(|J_\kZ|^{1/2})$ w.r.t.~the $\kZ$ pairing, so that
  the  inner products on $\widetilde\kZ$  and $\kZ$ are related as follows: 
 \[
 \langle z_1, z_2\rangle_{\kZ} = \langle  |J_\kZ| z_1, z_2\rangle_{\widetilde\kZ}.
 \]
 To define $\tilde\kX$ we proceed as follows:
 We write
 \[
 J_\kZ= \int_{\lambda \in \spec( J_\kZ)} \lambda \d \P_{\lambda}
 \]
 where $\d \P$ is the projection valued spectral measure of $J_\kZ$, see e.g., \cite[Theorem VIII.8]{ReedSimon}, noting that $\i J_\kZ$ is self-adjoint on $\kZ^\C$.
Polar decomposition gives   
\begin{equation}\label{e.polarDec}
J_{\kZ}= J_{\tilde\kZ} |J_\kZ| .
\end{equation}
 Here
 \[
   |J_\kZ| = \int_{ \i\omega \in \spec( J_\kZ)} |\omega| \d \P_{\i \omega}
 \]
 is  positive, self-adjoint and densely defined on $\kZ $ and all three operators commute. 
 Note that $ J_{\tilde\kZ}$ is by construction both skew-symmetric and unitary.
 Hence, $\spec( J_{\tilde\kZ}) = \{\i, -\i\}$. 
 
 We will now first find spaces $\kX$ and $\kY\simeq \kX$ of $\kZ = \kX\oplus \kY$ such that
 $ J_{\tilde\kZ}$, when restricted to $\kZ \subseteq\tilde\kZ$ takes the form
 \eqref{e.JTildeZ}. Then  we define 
 $\tilde \kX$ and $\tilde \kY\simeq \tilde\kX$ as  dual spaces of $\kX_{1/2} = D(|J_\kZ|^{1/2}) \cap \kX$  and $\kY_{1/2} = D(|J_\kZ|^{1/2}) \cap \kY$  w.r.t.~the $\kX$ and $\kY$ inner product respectively and conclude that  \eqref{e.JTildeZ} also holds on $\tilde \kZ$, the closure
 of $\kZ$ in the $\tilde \kZ$-norm.
 
 To define $\kX$, let  $\P_\pm$ be the orthonormal spectral projectors 
 onto the eigenspaces   $\kZ_\pm = \P_\pm \kZ^\C$  of $J_{\tilde\kZ}$ to its
 eigenvalues $\pm \i$.  Then $\kZ_+\simeq \kZ_-$ and so there is an isomorphism  $\iota:\kZ_-\to \kZ_+$ which we define as follows: let
 $\{e_j, j \in I\}$ be  an orthonormal basis for $\kZ_+$. Then
 $\{\bar e_j, j \in I\}$ is an orthonormal basis for $\kZ_-$
 and we define $\iota: \kZ^\C\rightarrow \kZ^\C$ by   
 \begin{align}
 \iota z = \sum_{j \in I }  \langle z, e_j \rangle_{\kZ^\C} \bar e_j  + 
 \langle z, \bar e_j \rangle_{\kZ^\C}  e_j.\label{e.iota}
 \end{align}
   Note that $\iota e_k = \bar e_k$, $\iota\bar e_k = e_k$ so that $\iota^2 = \id$ and $\iota^* =\iota$. (Technically, the isomorphism $\iota:\kZ_-\to\kZ_+$ above is the restriction of \eqref{e.iota} to $\kZ_-$.)
 We  set $\P_x = \frac{1}{ 2} ( \id+ \iota)$ and 
  $\P_y = \frac{1}{2}(\id-\iota)$. We then have  $\P_x+\P_y = \id$ and
 $\P_x\P_y =0$.
 Moreover we readily check  that $\kX\bot \kY$, 
 where  $\kX =\P_x  \kZ $ and $ \kY= \P_y \kZ$.
  Finally $\kX \simeq \kY$ because $\i(\P_+-\P_-)$
 is an isomorphism mapping $\kX$ to $\kY$ and $\kY$ to $\kX$. This follows from
 the fact that 
 \[
 \P_x\i(\P_+-\P_-) \P_x =\P_y \i(\P_+-\P_-)\P_y=0
 \]
 which is straightforward to check, using that  $\iota \P_\pm \iota = \P_\mp$ and  that $\iota \P_\pm = \P_\mp\iota$.

It remains to prove that the
symplectic structure operator $J_\kZ$ transforms to the operator $J_{\tilde\kZ}$  on $\tilde\kZ$, i.e.,  that \eqref{e.HamInf} is Hamiltonian on $\tilde\kZ$ with symplectic structure operator $J_{\tilde\kZ}$. This holds true  provided that
 \begin{equation*}
  J_{\tilde\kZ } \nabla_z^{\tilde\kZ }
 =  J_{\kZ} \nabla_z
 \end{equation*}
  and  this follows from 
 \begin{align*}
\nabla_z^{\widetilde\kZ} =  |J_\kZ|\nabla_z^{\kZ}
\end{align*}
which can be proved as follows:
 Let $f:\kZ\rightarrow \R$ be differentiable so that
 \begin{align*}
  \d f(z)(\delta z) = \langle \nabla_z^{\kZ} f(z),\delta z\rangle_{\kZ}.
 \end{align*}
Since $\langle z_1,z_2\rangle_{\kZ} = \langle |J_\kZ|z_1,z_2\rangle_{\widetilde\kZ}$, we have
\begin{align*}
 \langle \nabla_z^{\kZ} f(z),\delta z\rangle_{ \kZ}= \langle |J_\kZ|\nabla_z^{\kZ} f(z),\delta z\rangle_{\widetilde\kZ},
\end{align*}
and therefore 
\begin{align*}
 \nabla_{z}^{\widetilde\kZ} f(z) = |J_\kZ|\nabla_z^{\kZ} f(z).
\end{align*}
The generating function we consider is then
 \begin{align*}
  G(u,v_+,x,y_+) =  \langle x, y_+ \rangle_{\tilde\kX}  + \epsilon^{-1}\langle u, v_+\rangle +g(u,v_+,x,y_+),
\end{align*}
with  
\begin{align*}
g(u,v_+,x,y_+)&=\langle  J_{\widetilde\kZ}\zeta(u,v_+),(x,y_+) \rangle_{\widetilde\kZ} = \langle  |J_\kZ|^{-1}   J_{\widetilde\kZ}\zeta(u,v_+),(x,y_+) \rangle_\kZ \nonumber
\\
&  =- \langle J_{\kZ}^{-1}\zeta(u,v_+),(x,y_+) \rangle_\kZ,
\end{align*}
using \eqref{e.polarDec}.
The corresponding symplectic transformation defined as
 \[
 x_+ = \nabla_{y_+}^{\tilde\kX} G,\quad y=\nabla_{x}^{\tilde\kX}G, \quad u_+ =\epsilon \nabla_{v_+} G,\quad   v=\epsilon \nabla_{u} G 
 \]
 then gives the desired result.
\epr

\begin{remark}\label{r.genFct}\rm
In the following we will illustrate the abstract construction used in Lemma \ref{l.GenFct} by re-considering our two examples from Section \ref{ss.Examples}. 

In the case of the nonlinear Schr\"odinger equation, see Section \ref{ssec:nse},
we have $\kZ = \kH_1(\S^1;\R^2)$ and $J_\kZ = (1-\partial_s^2)J$  (with $J$ the standard
symplectic structure operator, { see \eqref{e.JL2}}) so that 
$\kZ_{1/2} = D_{\kH_1(\S^1;\R^2)}( |\partial_x|) = \kH_2(\S^1;\R^2)$ and $\tilde \kZ = \kZ_{-1/2}$,
the dual space of $\kZ_{1/2}$ w.r.t. the $\kZ$ inner product is $\tilde\kZ = \kL^2(\S^1;\R^2)$.
In this case $\kX = \kH_1(\S^1;\R)$ and $\tilde \kX = \kL^2(\S^1;\R)$.

In the case of the semilinear wave equation, see Section \ref{ssec:swe},  we have
$\kZ = \kH_1(\S^1;\R) \times \kL^2(\S^1;\R)$, $\kZ_{1/2} = \kH_{1.5}(\S^1;\R) \times \kH_{0.5}(\S^1;\R)$ 
and $\tilde\kZ = \kH_{0.5}(\S^1;\R)\times \kH_{-0.5}(\S^1;\R)$.  In this case let
\[
e_{\pm k}(s) = \frac{\e^{\i k s}}{2\sqrt \pi} { \frac{\pm 1}{\sqrt{k^2+1}} \choose \i},\quad k \in \N_0.
\]
Then $\kZ_+$ is spanned by $e_k$  and $\kZ_-$ by $e_{-k}$, $k \in \N_0$.
Moreover 
\[
e_k^r(s)= \sqrt 2\Re e_k(s) = \frac{1}{ \sqrt{2\pi}} {\frac{\cos ks}{\sqrt{k^2+1}} \choose -\sin ks},\quad
e_k^i(s)= \sqrt 2\Im e_k(s)  = \frac{1}{ \sqrt{2\pi}} {\frac{\sin ks}{\sqrt{k^2+1}} \choose \cos ks},
\]
and so one  copy of $\kX$ has orthonormal basis  $\{e_k^r, k \in \N_0\}$, the other, corresponding to $\kY$,
has orthonormal basis  $\{e_k^i, k \in \N_0\}$, both endowed with the $\kZ$ inner product. Moreover the two isomorphic copies of  $\tilde\kX$ are again spanned by $e_k^r$ and $e_k^i$, $k \in \N_0$, respectively, but now endowed with the $\tilde \kZ$ inner product.  Note that the decomposition $z=(x,y)$ for the semilinear wave equation used
 in the definition of the symplectic transformation, see \eqref{e.G-psi},
  is such that $x\in \kX$, $y \in \kY$, and hence not the natural decomposition
 where $\kZ = \kH_1(\S^1;\R)\times\kL^2(\S^1;\R)$ and $x \in \kH_1(\S^1;\R)$, $y \in \kL^2(\S^1;\R)$.  In this example we can identify $\kX$ with the space of square summable sequences $\ell_2(\N_0;\R)$ by identifying  each $x \in \kX$
with its  component vector
 with respect to the orthonormal basis  $\{e_k^r, k \in \N_0\}$, and similarly for $\kY$.
 We can then identify $\tilde\kX$ with the space of sequences endowed with inner product with weight $1/(k^2+1)^{1/4}$ for the $k$th component. 
 
It is important to highlight that the space $\tilde \kZ$ is only used in the proof of the above lemma for the construction of the generating function. We will not refer to it further.
\end{remark}

\subsection{Symplectic transformations}\label{ss.symplTrafo}
In this section we prove that the symplectic  transformation $(w_+,z_+)\mapsto \Psi(w_+, z_+) = (w,z)$   defined  
implicitly by the equations \eqref{e.G-psi} is well-defined. 
We will write the transformation $\Psi$ as 
\begin{align}\label{e.alphaBeta}
 z= {z_+ +  \psi^z(w_+,z_+) =} z_++\zeta(u,v_+),\quad w=w_+ +\epsilon \psi^w(w_+,z_+),
\end{align}
with $z_+=(x_+,y_+)$ and $w_+=(u_+,v_+)$. We also define
$\psi = (\epsilon\psi^w, \psi^z)=\Psi-\id$.
As before, let $\zeta(w)=(\zeta^x(w),\zeta^y(w))$  
be  the solution of \eqref{e.xyeqn}.
In the following we let $\kZ^\C_{-1}$ be the dual to $\kZ^\C_{1}$ with respect to the pairing $\langle z_1, z_2\rangle = \langle z_1, \bar z_2 \rangle_{\kZ^\C}$. 

\begin{lemma}\label{l.sglem}
 Assume (H0-H3)  for $H(w,z)$. Let $\xi>0$ be such that $\nu-\xi>0$ and $\sigma-\xi>0$ and assume 
 \be\label{e.xiCond}
 \xi \geq \max( 2K\delta, 8 \epsilon  C_\kS), \quad 0\leq \epsilon\leq 1/4
 \ee
 where $\delta = \Wert r\Wert_\nu$ satisfies \eqref{e.deltaCond} and, as before,
 $C_\kS= \Vert z\Vert_{\sigma}$.
 Then $(w_+,z_+)\mapsto \Psi(w_+,z_+) =  (w,z)$ 
 is an analytic symplectic transformation from $(\kV+\i(\nu-\xi))\times (\kS+\i(\sigma-\xi))$ to 
  $(\kV+\i(\nu-\xi/2))\times (\kS+\i(\sigma-\xi/2))$. Moreover
  $\psi^z$ is analytic from $(\kV+\i(\nu-\xi))\times (\kS+\i(\sigma-\xi))$ into
  $\kZ_{1}^\C$, ${\psi^w}$, $\psi^z$ are also analytic in $z \in \kB^{\kZ_{-1}^\C}_{C_\kS-\xi}(0)$  and  
\begin{align}
   \Vert {\psi^z} \Wert_{\nu-\xi,\sigma-\xi}  \le \frac{\xi}{2},
   \quad  \epsilon \Vert {\psi^w} \Vert_{\nu-\xi,{\sigma-\xi}} \le \frac{\xi}{2},\label{e.zwpest}
\end{align}
where the norm $\Vert \cdot \Wert$ is defined in (H3).
\end{lemma}
\bpr
 The functions $\psi^z =\psi^z(w_+, z_+) $, $\psi^u = \psi^u(w^+, z^+)$
 and $\psi^v=\psi^v(w^+, z^+)$ are defined via the equations
    \begin{subequations}\label{e.betauv}
\begin{align}   
 \psi^z  & =    \zeta(u_++ \epsilon\psi^u ,v_+), \\
 \psi^u  & =\langle J_{\kZ}^{-1} \partial_{v} \zeta(u_++\epsilon \psi^u,v_+),
 \begin{pmatrix}
         x_++\zeta^x(u_++\epsilon \psi^u,v_+)\\
       y_+
      \end{pmatrix}\rangle,\label{e.psiu}\\
       \psi^v  & = -\langle J_{\kZ}^{-1} \partial_{u} \zeta(u_++\epsilon \psi^u,v_+),
       \begin{pmatrix}
          x_++\zeta^x(u_++\epsilon \psi^u,v_+)\label{e.psiv}\\
   y_+
    \end{pmatrix}\rangle.                                                                                                                                      
\end{align}
\end{subequations}
When $\epsilon\neq 0$ the function $\psi^u$ is implicitly defined via \eqref{e.psiu}. 
We rewrite this equation as $\phi= \Pi(\phi, w_+, z_+)$ where $\phi = \psi^u$ and show that $\Pi$ is a contraction
on $\kB^{\R^{d}}_{\eta}(0)$ where $\eta = \frac\xi{2\epsilon}$ and $d=d_\kW$.
Note that $\Pi(\phi, w_+, z_+)$ takes the form
\[
\Pi(\phi, w_+, z_+) = \Pi_0(\phi, w_+) + \langle \Pi_1(\phi, w_+ ), z_+\rangle
\]
and hence is affine in $z_+$. Moreover since $J_{\kZ}^{-1} \in \kE(\kZ;\kZ_1)$ we see that
$\Pi_1(\phi, w_+,z_+ ) \in \kZ_1^\C$ and hence that $\Pi$ is affine in $z_+ \in \kZ_{-1}^\C$.

We estimate using Lemma \ref{l.analimpl}, Cauchy estimates (Lemma \ref{l.cest}) and \eqref{e.xiCond} that
\begin{equation}
\|\Pi(\phi, \cdot,\cdot)\|_{\nu-\xi,\sigma-\xi} \leq  \|\partial_v\zeta\|_{\nu-\xi/2} C_\kS \leq 2   C_\kS K\delta/\xi 
\leq  C_\kS
\leq \frac{\xi}{8\epsilon}.
\label{e.PiEst}
\end{equation}
Hence $\Pi(\cdot, w_+, z_+)$ maps $\kB^{\R^{d}}_{\eta}(0)$ to itself 
for all $w_+ \in \kV +\i( \nu-\xi)$, $z \in \kS+\i(\sigma-\xi)$.
Moreover,
\begin{align*}
\partial_\phi \Pi(\phi, w_+, z_+) & =  \epsilon\langle J_{\kZ}^{-1} \partial^2_{uv} \zeta(u_++\epsilon \phi ,v_+),
 \begin{pmatrix}
         x_++\zeta^x(u_++\epsilon \phi,v_+) \\
       y_+
      \end{pmatrix}\rangle \\
      & + \epsilon \langle J_{\kZ}^{-1}\partial_{v} \zeta(u_++\epsilon \phi,v_+),
  \P_x\partial_u\zeta(u_++\epsilon \phi,v_+) \rangle
\end{align*}
and so,  again by Lemma \ref{l.analimpl}, Cauchy estimates and  \eqref{e.xiCond}
\begin{align}
\|\partial_\phi\Pi(\phi, \cdot, \cdot)\|_{\nu-\xi, \sigma-\xi} 
&\leq \epsilon \Vert \partial^2_{uv} \zeta\Vert_{\nu-\xi/2} C_\kS + 
 \epsilon\Vert \partial_{u} \zeta\Vert_{\nu-\xi/2} \Vert \partial_{v} \zeta\Vert_{\nu-\xi/2} 
\\
 & \leq \epsilon \left( \frac{8 K\delta C_\kS}{\xi^2}  +\left(\frac{2 K\delta}{\xi}\right)^2\right) \notag \\
&\leq  \frac{4  \epsilon C_\kS}{\xi} + \epsilon\leq \frac 12+\epsilon \leq \frac 34.
\label{e.PiContract}
\end{align}
Hence $\Pi$ is a contraction on $\kB^{\R^{d}}_{\eta}(0)$, and so by the contraction mapping theorem it has a unique fixed point $\psi^u(w_+, z_+)  := \phi(w_+, z_+)$ which
is analytic in $(w_+, z_+)  \in (\kV +\i( \nu-\xi)) \times ( \kS+\i(\sigma-\xi))$.
Note that estimates \eqref{e.PiEst}, \eqref{e.PiContract} also hold for $z_+ \in \kB^{\kZ_{-1}^\C}_{C_\kS-\xi}(0)$ which proves that $\psi^u$ is also analytic in  $z_+ \in \kB^{\kZ_{-1}^\C}_{C_\kS-\xi}(0)$. By the chain rule, the same applies for $\psi^v$ and
$\psi^z$.
\epr

\begin{cor}\label{c.alphaest}
 The estimate for $\Wert {\psi^z}\Wert_{\nu-\xi,{\sigma-\xi}}$ in \eqref{e.zwpest} can be improved to
 \begin{subequations}
\begin{align}\label{e.psizEst}
 \Vert {\psi^z}\Wert_{\nu-\xi,{\sigma-\xi}}\le \Wert \zeta\Wert_{\nu}\leq K \delta,
\end{align}
and we also have
\begin{equation}
 \Vert {\psi^w}\Vert_{\nu-\xi,{\sigma-\xi}}\le  4 C_\kS  K\delta/\xi.
 \label{e.psiwEst}
 \end{equation}
 \end{subequations}
\end{cor}
\bpr
 From Lemma \ref{l.sglem} we know that
 $(u\vert_{\kV + \i(\nu-\xi),\kS+\i(\sigma-\xi)},v_+)\subset \kV+\i(\nu-\xi/2)$. Hence,
\begin{align*}
\Vert {\psi^z} \Wert_{\nu-\xi,{\sigma-\xi}}= \Vert \zeta(u,v_+)\Wert_{\nu-\xi,\sigma-\xi}
\le \Wert \zeta\Wert_{\nu-\xi/2}\le \Wert \zeta\Wert_{\nu}.
\end{align*}
The other estimate follows as in the proof of Lemma \ref{l.sglem} for both $\psi^u$ and $\psi^v$
upon replacing $\eta$ by $2 C_\kS  K\delta/\xi$. 
\epr

Let  
\begin{align*}
 \psi_0 = \psi\vert_{z_+=0},\quad \psi_1 =\partial_z\psi\vert_{z_+=0},\quad \psi_2 =\partial_z^2\psi\vert_{z_+=0},
\end{align*}
and $\psi_i=(\psi_{i}^z,\psi_{i}^w)$, $i=0,1,2$. 
Then
\begin{subequations}
\begin{align}\label{e.beta0-est}
 \epsilon \Vert \psi^w_0 \Vert_{\nu-\xi} &= \Vert w_+-w(w_+,0)\Vert_{\nu-\xi} 
 \le \frac{\xi}2.
\end{align}
In the next estimate we also interpret $\psi_1^w$  and $\psi_2^w$ through 
\begin{align*}
\langle \psi_1^w,z\rangle = \partial_{z_+} \psi^w\vert_{z_+=0}z,&\quad  
\langle \psi_2^w z_{1},z_2\rangle = \partial_{z_+}^2 \psi^w\vert_{z_+=0}z_1 z_2,
\end{align*}
for all $z_1,z_2\in \kZ^{\C}$. 
Using Cauchy-estimates on $\psi^w_1=\partial_{z} \psi^w(w_+,0)$ and $\psi^w_2 = \partial_{z}^2 \psi^w(w_+,0)$, noting that $\psi^w$ is analytic in $z \in \kZ_{-1}^\C$ 
we obtain
\begin{align}\label{e.beta12-est}
 \Wert \psi^w_1\Wert_{\nu-\xi} \le \kappa^{-1} {\Vert\psi^w\Vert_{\nu-\xi,\sigma-\xi}},\quad
\Vert\psi^w_2\Wert_{\nu-\xi} \le {2\kappa^{-2}\Vert \psi^w \Vert_{\nu-\xi,\sigma-\xi}},
\end{align}
where $0<\kappa \leq \sigma-\xi$. 
Moreover 
\begin{align}\label{e.psiz1-est}
 \Vert \psi^z_1 \Wert_{\nu-\xi}
 \leq\frac{  K \delta}{\kappa}  \leq \sigma/\kappa, \quad \Vert \psi^z_2 \Wert_{\nu-\xi}
 \leq\frac{2  K \delta}{\kappa^2}.
 \end{align}
\end{subequations}

The fact that $(w_+,z_+)\mapsto (w,z)$ is well-defined with domain $(\kV+\i(\nu-\xi)) \times (\kS+\i(\sigma-\xi))$ 
and co-domain $(\kV+\i(\nu-\xi/2)) \times (\kS+\i(\sigma-\xi/2))$ was crucial here and will be in the following. 
The $\nu-\xi/{2}$ and $\sigma-\xi/2$ {terms} in the co-domains allow for a step of $\xi/2$ to 
apply Lemma \ref{l.cest} to estimate derivatives on $(\kV+\i(\nu-\xi/2)) \times (\kS+\i(\sigma-\xi/2))$ 
by function values on the larger domain $(\kV+\i\nu) \times (\kS+\i\sigma)$, cf. the Cauchy estimate 
\eqref{e.cest}.
This introduces a factor of $2\xi^{-1}$.  

\subsection{Iterative lemma}\label{ss.ItLemma}
We are now ready to state and prove an Iterative Lemma  which will is main ingredient
of the proof of Theorem \ref{t.main}.

\begin{lemma}\label{l.itHam}
 (\textbf{The Iterative Lemma for Hamiltonian systems}) 
Assume   (H0-H3)   for the Hamiltonian $H$, 
but relax the bounds on $\partial_w h$ and $F$ slightly, so 
 that $h$,  $a$, $r$, $f$ and $F$ satisfy
 \begin{subequations}
\begin{align*}
&\Vert h\Vert_{\nu} \le C_h, \,\Vert \partial_w h\Vert_{\nu-\xi/2}\le C_h^\prime,\\
&\Vert a \Wert_{ \nu}\le C_a,\, \Wert (L +a)^{-1}\Wert_{\nu}\le {K}/{2},\\
&\Vert f\Vert_{\nu, \sigma}\le C_f , \Vert F\Wert_{\nu, \sigma-\xi/2}\le C_F, \\
&\Wert r\Wert_{\nu}\le \delta.
\end{align*}
Here       $\delta>0$ satisfies \eqref{e.deltaCond} and 
\begin{equation}\label{e.ItLeHam-conds}
\min(\nu,\sigma) > \xi\ge \max\{ 8 C_\kS \epsilon ,2K\delta\}, ~~ 0\leq  \epsilon\leq \frac 14,
\end{equation}
where $C_\kS=\|z\|_{\sigma}$ as before.
 Then the symplectic transformation $\Psi:(w_+,z_+)\mapsto (w,z)$ from Lemma \ref{l.sglem} mapping 
 $(\kV+\i(\nu-\xi))\times (\kS+\i(\sigma-\xi))$ into $(\kV+\i(\nu-\xi/2)) \times (\kS+\i(\sigma-\xi/2))$ 
 transforms $H=H(w,z)$  into
\begin{align}
  H_+(w_+,z_+) ={h_+}(w_+)+ \langle r_+(w_+), z_+\rangle
   +\frac12 \langle (L+ a_+(w_+)) z_+,  z_+\rangle
   +{f_+}(w_+,z_+).\label{e.Hamp}
\end{align}
 Here  $f_+ = \kO(\|z_+\|^3)$,
 \begin{align*}
 a_+ \in \kC^\omega(\kV+\i(\nu-\xi);\kE(\kZ^{\C};\kZ^\C_{1})),
\end{align*}
and
 \begin{align*}
 F_+ = \nabla_z f_+ \in \kC^\omega(( \kV+\i(\nu-\xi)) \times (\kS + \i (\sigma-3\xi/2); \kZ^\C_{1}).
\end{align*}
Moreover there is a constant $c$  which is increasing in  $C_h'$, $K$, $C_F$,  $C_a$, $\delta$,  $C_H$, $C_\kS$ and $1/\kappa$, with $ \kappa$ satisfying
 $0<\kappa\leq \sigma-\xi$,  and  depends   continuously on those constants   only
such that 
\begin{align}\label{e.deltaHam-est}
 \delta_+ =\Wert r_+\Wert_{\nu-\xi} \le c\frac{\epsilon\delta }{\xi},
\end{align}
and 
\begin{align}\label{e.ham-it-ests}
\Vert {h_+}-h\Vert_{\nu-\xi}\le c \delta, &\quad
 \Vert \partial_{w_+}( h_+ - h)\Vert_{\nu- 3\xi/2} \le c\frac{\delta}{\xi},
\nonumber \\
\Vert {f_+}-f\Vert_{\nu-\xi,\sigma-\xi}\le c\frac{ \delta}{\xi},&\quad
 \Vert  {F_+} -F \Wert_{\nu-\xi,\sigma-3\xi/2} \le  c\frac{ \delta}{\xi^2}, \nonumber\\
 \Vert {a_+}-a\Wert_{\nu-\xi}\le c\frac{ \delta}{\xi}.
\end{align}
Furthermore,
\begin{align}\label{e.hamK+}
\Wert (L+ a_+)^{-1}\Wert_{\nu-\xi}\le {\frac{K_+}{2}} :=\frac{K}{2} 
+ c\frac{ \delta }{\xi},
\end{align} 
provided
\begin{equation}\label{e.K+}
c \delta < \xi.
\end{equation}
\end{subequations}
\end{lemma}

\bpr
We Taylor expand the new Hamiltonian 
$H_+(w_+,z_+)=H(w,z)$ around $z_+=0$ to put it into the form \eqref{e.Hamp} with
\begin{align*}		
 {h_+}(w_+) = H_+(w_+,0),\quad r_+(w_+) = \nabla_{z_+} H_+(w_+,0),\quad
\end{align*}
with
\[
{a_+}(w_+) = \partial_{z_+} \nabla_{z_+}H_+(w_+,0)-  L,
\quad
\]
and
\[ {f_+}(w_+, z_+)=H_+(w_+, z_+)-{h_+(w_+)}- \langle r_+(w_+),  z_+\rangle -\frac12  \langle (L+ a_+)z_+,  z_+\rangle.
\]
We  have for $w_+ \in \kV + \i(\nu-\xi)$
\begin{align}
|{h_+}(w_+) - h(w_+) |&= |H_+(w_+,0)  - H(w_+,0)|
= |H(w\vert_{z_+=0},z\vert_{z_+=0}) - H(w_+,0)|\notag\\
&\leq |H(w_+,0)-H(w\vert_{z_+=0},0) |\notag \\
& +|H(w\vert_{z_+=0},0) - H(w\vert_{z_+=0},\psi^z\vert_{z_+=0})| ,\nonumber\\
&\leq\epsilon \max_{s \in [0,1]}\|\partial_w h \circ (w_++ s\epsilon \psi_0^w)\|_{\nu-\xi}  \|\psi^0_w\|_{\nu-\xi}\notag\\
& +\max_{s\in [0,1]} \| \partial_z H \circ(w_++\epsilon \psi^w_0, s \psi^z_0) \|_{\nu-\xi} \|\psi^z_0\|_{\nu-\xi},\nonumber\\
&\leq\epsilon\|\partial_w h\|_{\nu-\xi/2}  \|\psi_0^w\|_{\nu-\xi}+ \| \partial_z H \|_{\nu-\xi/2,\xi/2} \|\psi^z_0\|_{\nu-\xi},\nonumber\\
&\leq  C_h'  \epsilon\|\psi_0^w\|_{\nu-\xi}+  \frac{ C_H}{\kappa} \| \psi^z_0\|_{\nu-\xi}
\leq 4  C_\kS C_h'  K\frac{\epsilon \delta}{\xi} +  \frac{ C_H K }{\kappa} \delta \notag \\
& \leq 
\left( C_h' + \frac{ C_H }{\kappa}\right)  K\delta \label{e.H+w+}.
\end{align}
Here we used the mean value theorem, Lemma  \ref{l.sglem}  and Corollary \ref{c.alphaest}, and in the last inequality we used that $\xi\geq 8 C_\kS  \epsilon$ by
\eqref{e.ItLeHam-conds}.

Moreover, note that \eqref{e.H+w+} and a Cauchy estimate give 
\begin{align*}
  \Vert \partial_w ({h_+}-h)\Vert_{\nu-3\xi/2} 
&\le  \frac{2}{\xi} \Vert {h_+-h}\Vert_{\nu-\xi}   \leq
 \left( C_h' + \frac{ C_H  }{\kappa}\right)\frac{2 K \delta}{\xi}.
\end{align*}
Let $\psi_0 = (\epsilon\psi^w_0, \psi^z_0)$. 
Then, using \eqref{e.alphaBeta}, \eqref{e.beta12-est} we obtain
 for $w_+ \in \kV + \i(\nu-\xi)$
\begin{align}
\Wert r_+(w_+)\Wert&=\Wert\nabla_{z_+} H_+(w_+,0)\Wert
=\Wert\nabla_{z_+} H(w,z)\vert_{z_+=0}\Wert \nonumber\\
&\leq 
\| \partial_z H(w\vert_{z_+=0},z\vert_{z_+=0})\frac{\partial z}{\partial z_+}\vert_{z_+=0}
 \|_{\kE(\kZ_{-1}, \C)} \notag\\
 & +\|\partial_w H(w\vert_{z_+=0},z\vert_{z_+=0}) \frac{\partial w}{\partial z_+}|_{z_+=0}\|_{\kE(\kZ_{-1}, \C)} \nonumber\\
 & \le \Wert\nabla_z H(w\vert_{z_+=0},\zeta(u\vert_{z_+=0}, v_+))\Wert (1+\Wert \psi_1^z\Wert_{\nu-\xi} )
 \notag \\
 & +\epsilon\|\partial_w H( w\vert_{z_+=0}, \psi^z_0)\| \Wert \psi_1^w \Wert_{\nu-\xi}. \label{e.rEst}
 \end{align}
  We estimate,  with \eqref{e.zwpest}, that 
 \[
\| (\partial_w H)\circ (w\vert_{z_+=0}, \psi^z_0) \|_{\nu-\xi} 
\leq \| \partial_w H \|_{\nu-\xi/2, \xi/2}.
 \]
 Then, using the mean value theorem and Cauchy's estimate,  gives
\begin{align*}
\| \partial_w H \|_{\nu-\xi/2, \xi/2  }  &\leq
\|  \partial_w h\|_{\nu-\xi/2} + \| \partial^2_{wz} H\|_{\nu-\xi/2,\xi/2} \cdot \frac{\xi}{2}\notag\\
 &  \leq
C_h' +\frac{ \|\partial_w H\|_{\nu-\xi/2, \sigma}}{\kappa} \cdot\frac{\xi}{2}
\leq C_h' + \kappa^{-1}C_H.
 \end{align*}
 Moreover, since $\partial_z H(w,\zeta(w))=0$,
  using Cauchy's estimate and \eqref{e.psiwEst} we obtain
 \begin{align*}
 \Wert\nabla_z H&(w\vert_{z_+=0},\zeta(u\vert_{z_+=0}, v^+))\Wert
 \\
 & =   \Wert\nabla _z H(w\vert_{z_+=0},\zeta(u\vert_{z_+=0}, v^+))- \nabla_z H(w\vert_{z_+=0},\zeta(w\vert_{z_+=0}))\Wert \\
&\leq  \max_{\substack { \Wert z\Wert \leq\xi/2}} \|\partial_z \nabla_z H(w\vert_{z_+=0},z) \|_{\kE(\kZ_1; \kZ_1)} \cdot\Wert \zeta(u\vert_{z_+=0}, v^+)  - \zeta(w\vert_{z_+=0})\Wert \\
& \leq \Wert L +  \partial_z\nabla_z V \Wert_{\nu-\xi/2, \xi/2} \Wert \partial_v \zeta\Wert_{\nu-\xi/2}\epsilon \| \psi^v_0\|_{\nu-\xi}\\
&  \leq( \Wert L\Wert+ \Vert a \Wert_{\nu-\xi/2} +  \Vert \partial_z F \Wert_{\nu-\xi/2, \xi/2} ) \cdot \frac{2K \delta}{\xi} \cdot \frac{4 C_\kS   K \delta\epsilon}{\xi} \\ 
& \leq (\Wert L\Wert+ C_a +\kappa^{-1}C_F) \frac{4 C_\kS  K  \delta\epsilon}{\xi} .
 \end{align*}
Here we denote
\[
 \Vert \partial_z F \Wert_{ \nu,\sigma} := \sup_{\substack{w \in \kV+\i \nu,\\ z \in \kS+\i \sigma}} \|\partial_z F(w,z)\|_{\kE(\kZ;\kZ_1)}, 
 \]
and, defining $\Vert \partial_z \nabla_z V \Wert_{ \nu,\sigma}$,  analogously we use that 
 \[
 \Wert \partial_z \nabla_z V \Wert_{ \nu,\sigma} \leq \Vert \partial_z \nabla_z V \Wert_{ \nu,\sigma}   \leq \Vert a \Wert_{ \nu}+
  \Vert \partial_z \nabla_z F \Wert_{ \nu,\sigma} .
\]
 Plugging these estimates into \eqref{e.rEst}, using 
\eqref{e.beta12-est} and \eqref{e.psiz1-est},  we obtain
 \begin{align*}
\Wert r_+\Wert_{\nu-\xi} & 
\leq  (\Wert L \Wert+ C_a +\kappa^{-1}C_F)  \frac{ 4  C_\kS K\delta\epsilon}{\xi}\cdot 
 (1+\Wert\psi_1^z\Wert_{\nu-\xi} )
 +\epsilon (C_h' + \kappa^{-1} C_H) \Wert \psi_1^w \Wert_{\nu-\xi}\\
 &\leq  (\Wert L \Wert+ C_a +\kappa^{-1}C_F)  \frac{  4 C_\kS K\delta\epsilon}{\xi}\cdot (1+ \sigma\kappa^{-1})
 + (C_h' +  \kappa^{-1}C_H) \frac{4\epsilon C_\kS K\delta}{\xi \kappa} \\
 & \leq 
 c \frac{\epsilon K \delta}{\xi}.
\end{align*}

For ${A_+} =L + a_+$ we have {using \eqref{e.Hamp}},
that $\langle A_+ z_1, z_2 \rangle = \partial^2_{z_+} H_+(w_+,0)  z_1 z_2$
where 
\begin{align*}
 \partial^2_{z_+} H_+(w_+,0) 
&  =
\partial_{z_+}( \partial_z H(w,z) \partial_{z_+}z
 + \partial_w H(w,z ) \partial_{z_+} w)\vert_{z_+=0} \notag\\
& = \partial_z^2 H (\partial_{z_+}z)^2 +  \partial_z H \partial^2_{z_+}z
+\partial_w^2H (\partial_{z_+} w)^2+ \partial_w H \partial_{z_+}^2 w \\
&  +2 (\partial^2_{zw} H \partial_{z_+} z \partial_{z_+} w)_{\rm sym} \notag\\
&= \partial^2_{z}H(\id+\psi_1^z)^2+\partial_z H \psi_2^z+\epsilon^2 \partial^2_w H (\psi^w_1)^2+  \epsilon  \partial_w H \psi^w_2\\
& 
+2\epsilon  (\partial^2_{zw} H (\id+\psi_1^z)\psi^w_1)_{\rm sym}\\
&= (L+ \partial_z^2 V) (\id+\psi_1^z)^2+\partial_z H \psi_2^z
+\epsilon^2 \partial^2_w H (\psi^w_1)^2+  \epsilon  \partial_w H \psi^w_2
\\
& +2\epsilon  (\partial^2_{zw} V (\id+\psi_1^z)\psi^w_1)_{\rm sym}.
\end{align*}
Here  we have used \eqref{e.alphaBeta} and  
for any bilinear form  $M:\kZ^\C\times \kZ^\C \to \C$ we define $M_{\rm sym}$ to be its
symmetrization:
\[
M_{\rm sym}(z_1, z_2)  = \frac 12( M(z_1, z_2) + M(z_2, z_1));
\]
 all derivatives of $H$ and $V$ are evaluated at 
$({w}\vert_{z_+=0},z\vert_{z_+=0})$ and all derivatives w.r.t.~$z_+$ of $z$ and $w$ are evaluated at $z_+=0$.
Hence
\begin{align*}
\langle \cdot, (a_+(w_+)-a(w_+)) \cdot\rangle &= \partial_z^2 V - \partial_z^2 V (w_+,0)
+  (\partial_z^2 V+L) (2 +\psi_1^z)\psi_1^z +\partial_z H \psi_2^z \\
& 
+\epsilon^2 \partial^2_w H (\psi^w_1)^2+  \epsilon  \partial_w H \psi^w_2
+2\epsilon ( \partial^2_{zw} V (\id+\psi_1^z)\psi^w_1)_{\rm sym}.
\end{align*}
 Therefore, using
 Lemma \ref{l.analimpl}, \eqref{e.psiwEst}, \eqref{e.beta0-est},   \eqref{e.beta12-est}, and  \eqref{e.psiz1-est}  we obtain
\begin{align} 
 \Vert{a_+}(w_+)-a(w_+) \Wert_{\nu-\xi}& = \Vert\partial_{z}\nabla_z V(w\vert_{z_+=0}, z\vert_{z_+=0}) - \partial_{z}\nabla_z V(w_+,0)\Wert_{\nu-\xi}\notag\\
 & +( \Wert \partial_z\nabla_z V\Wert_{\nu-\xi/2, \xi/2} +\Wert L \Wert)(\Wert \psi^z_1\Wert_{\nu-\xi} + 2) \Wert \psi^z_1\Wert_{\nu-\xi} 
  \notag\\
 &
   +\Vert \partial_z H \psi_2^z \Wert_{\nu-\xi}+  \epsilon^2 \| \partial^2_w H\|_{\nu-\xi/2, \xi/2} \Wert\psi^w_1\Wert^2_{\nu-\xi}
 \notag\\
 &  +\epsilon \|\partial_w H\|_{\nu-\xi/2, \xi/2}\Wert \psi^w_2\Wert_{\nu-\xi}
\notag\\
& +2\epsilon \Vert \partial_{w}\nabla_z V\Wert_{\nu-\xi/2, \xi/2} (1+\Vert\psi_1^z\Wert_{\nu-\xi})\Wert\psi^w_1\Wert_{\nu-\xi}\notag\\
& \leq    \Vert\partial_{z}\nabla_z V(w\vert_{z_+=0}, z\vert_{z_+=0}) - \partial_{z}\nabla_z V(w_+,0)\Wert_{\nu-\xi}\notag\\
& +(C_a + C_F \kappa^{-1} + \Wert L \Wert) \cdot (2+ \sigma\kappa^{-1})\frac{K\delta }{\kappa}+
\frac{C_H}{\kappa}\cdot \frac{2K\delta}{\kappa^2}\notag\\
&  +  \| \partial^2_w H\|_{\nu-\xi/2, \xi/2}\left( \frac{ 4\epsilon K C_\kS  \delta}{\xi\kappa}\right)^2     
 +\| \partial_w H\|_{\nu-\xi/2, \xi/2} \frac{ 8\epsilon K C_\kS   \delta}{\xi\kappa^2}
\notag \\
& +  \Vert \partial_{w} \nabla_z V\Wert_{\nu-\xi/2, \xi/2} (1+\sigma\kappa^{-1}) \frac{ 8\epsilon K C_\kS \delta}{\xi\kappa}.
 \label{e.AEst}
 \end{align}
  Using    Cauchy's estimate we get
 \[
  \Vert \partial_{w} \nabla_z V\Wert_{\nu-\xi/2, \xi/2} \leq 
  \Vert \partial_{w} a z + \partial_w r + \partial_w F\Wert_{\nu-\xi/2, \xi/2}  \leq
\frac{2 (C_\kS C_a+ \delta  + C_F)}{\xi }
  \]
  and
  \[
   \| \partial_{w} H\|_{\nu-\xi/2, \sigma} \leq
   \frac{2C_H}{  \xi}, \quad
\| \partial^2_w H\|_{\nu-\xi/2, \xi/2}   \leq \frac{4 C_H}{\xi^2}.
 \]
 Furthermore, using the mean value theorem, Cauchy's estimate  and \eqref{e.psizEst}, \eqref{e.psiwEst} we obtain for $w_+ \in \kV + \i(\nu-\xi)$ 
 \begin{align*}
 \Vert\partial_{z}\nabla_z & V(w\vert_{z_+=0}, z\vert_{z_+=0}) - \partial_{z}\nabla_z V(w_+,0) \Wert
 \\
 &  \leq 
   \Vert\partial_{z}F(w\vert_{z_+=0}, z\vert_{z_+=0}) - \partial_{z}F(w\vert_{z_+=0}, 0)\Wert  \\
  &\quad +    \Vert\partial_{z}F(w\vert_{z_+=0}, 0) - \partial_{z}F(w_+,0)\Wert  + \Vert a(w\vert_{z_+=0}) - a(w_+)\Wert \\
& \leq \Vert\partial^2_{z}F\Wert_{\nu-\xi/2, \xi/2} \| \psi^z_0\|_{\nu-\xi} 
+  \epsilon\Vert\partial_{z}\partial_w F\Wert_{\nu-\xi/2, \xi/2} \| \psi^w_0\|_{\nu-\xi} 
\\
& \quad + \Vert\partial_w a\Wert_{\nu-\xi/2}\epsilon \| \psi^w_0\|_{\nu-\xi} \\
& \leq \frac{2 C_F}{\kappa^2}\cdot K\delta + \frac{2C_F}{\kappa \xi}\cdot\frac{4\epsilon  C_\kS K\delta}{\xi}
+ \frac{2C_a}{\xi} \cdot\frac{4\epsilon  C_\kS K\delta}{\xi}.
\end{align*}
  Plugging these    into \eqref{e.AEst} 
the estimate for $a_+-a$ in \eqref{e.ham-it-ests} then follows upon use of the conditions in \eqref{e.ItLeHam-conds}. Moreover  for  $A_+(w_+) =L + a_+(w_+)$ we have 
\[
\frac{2}{K} \Wert z\Wert \le \Wert A(w_+) z\Wert
\leq \Wert(a(w_+)-a_+(w_+) ) z\Wert+ \Wert A_+(w_+)z \Wert
\]
and so 
\[
 \Wert A_+(w_+)z \Wert \geq  \left(\frac{2}{K}  -  \Wert (a(w_+)-a_+(w_+)  \Wert \right) \Wert z \Wert
\]
and hence
\begin{align*}
\Wert A_+(w)^{-1}\Wert &\le 
\frac{K}{2}\left(1 - \frac{K c\delta }{2\xi}\right)^{-1} \leq \frac{K}2 \left(1+ \frac{c K \delta}{\xi}\right)
\end{align*}
for $ K c\delta<\xi$. Here we have used that  $(1-x)^{-1} = 1 + \frac{x}{1-x} \leq 1+2x$ if $x \in [0,\frac12]$,
with $x=\frac{c\delta K }{2\xi}$. Redefining $c$ to $\max(1, K^2)c$ then  verifies \eqref{e.hamK+} if \eqref{e.K+} holds true.

Finally
\begin{align*}
\Vert F_+ -F \Wert_{\nu-\xi, \sigma-3\xi/2}   \leq & 
\Vert\nabla_{z_+} H_+ -\nabla_{z_+} H \Wert_{\nu-\xi, \sigma-3\xi/2}
+ 
 \Wert r_+  - r  \Wert_{\nu-\xi} \\
 &  + \Vert a_+  - a\Wert_{\nu-\xi} C_\kS
\end{align*}
Using the estimates for $\Wert r_+-r\Wert_{\nu-\xi}$ and $ \Vert a_+  - a\Wert_{\nu-\xi}$
obtained above  we only need to estimate the following
\begin{align*}
\Vert\nabla_{z_+} & H_+ - \nabla_{z_+}H \Wert_{\nu-\xi, \sigma-3\xi/2} \leq 
\Vert(\id +\partial_{z_+} \psi^z)^*\nabla_z H \circ\Psi -\nabla_{z_+} H\Wert_{\nu-\xi, \sigma-3\xi/2}
\\
& +  \epsilon\Vert(\partial_{z_+}\psi^w)^* (\partial_w H  \circ \Psi)^* \Wert_{\nu-\xi, \sigma-3\xi/2}\\
& \leq \Vert(\nabla_{z} H)\circ \Psi   -\nabla_{z_+} H \Wert_{\nu-\xi, \sigma-3\xi/2}
+ \Vert(\partial_{z_+}\psi^z)^*(\nabla_{z} H\circ \Psi ) \Wert_{\nu-\xi, \sigma-3\xi/2}\\
&  +  \epsilon\Vert\partial_{z_+}\psi^w \Wert_{\nu-\xi, \sigma-3\xi/2} \Vert \partial_w H \Vert_{\nu-\xi/2, \sigma-\xi}\\
& \leq \Vert \nabla_z\partial_w H \Wert_{\nu-\xi/2, \sigma-\xi}
 \epsilon \Vert \psi^w\Vert _{\nu-\xi, \sigma-3\xi/2}\\
 & + 
\Wert \nabla_z \partial_z H   \Wert_{\nu-\xi/2, \sigma-\xi} \Vert\psi^z\Wert _{\nu-\xi, \sigma-3\xi/3} + \Vert \nabla_{z} H\Vert_{\nu-\xi/2, \sigma-\xi}\Vert \partial_{z_+}\psi^z\Wert_{\nu-\xi, \sigma-3\xi/2}\\
&  +  \epsilon\Vert\partial_{z_+}\psi^w \Wert_{\nu-\xi, \sigma-3\xi/2} \Vert \partial_w H \Vert_{\nu-\xi/2, \sigma-\xi}\\
 &\leq \Vert \partial_w r+  \partial_w a  z + \partial_w F \Wert_{\nu-\xi/2, \sigma-\xi}
 \frac{4\epsilon C_\kS  K \delta}{\xi} +  \Wert a + L+\partial_z F   \Wert_{\nu-\xi/2, \sigma-\xi}K \delta\\
& + \Vert  r+ Lz+ az +F\Vert_{\nu-\xi/2, \sigma-\xi}  \Vert \partial_{z_+}\psi^z\Wert_{\nu-\xi, \sigma-3\xi/2} 
+\frac{2C_H}\xi   \frac{4\epsilon C_\kS K \delta}{\xi} \frac{2}\xi
 \end{align*}
and so
\begin{align*}
\Vert\nabla_{z_+} & H_+ - \nabla_{z_+}H \Wert_{\nu-\xi, \sigma-3\xi/2} 
    \leq   (C_r+  C_a  C_\kS + C_F)\frac{2}{\xi}
 \frac{4\epsilon C_\kS  K \delta}{\xi}  \\
  &+ (C_a + \Vert L\Vert + 2C_F/\xi) K \delta +( C_r+( C_a  +\Vert L\Vert) C_\kS +C_F) \frac{2 K \delta}{\xi}
 +\frac{2C_H  K \delta}{\xi^2}\\
 & \leq c \delta/\xi^2,
\end{align*}
where we used that $H \circ \Psi = H_+$ in the first inequality and  the definition
$\psi: = (\psi^z, \epsilon \psi^w)$ in the third and fourth inequality.
In the third inequality we use the mean value theorem and in the fourth inequality and final inequality 
Cauchy's estimate together with \eqref{e.psizEst} and \eqref{e.psiwEst} and \eqref{e.ItLeHam-conds}. A similar estimate shows that $\Vert f_+-f\Vert_{\nu-\xi, \sigma-\xi}\leq c \delta/\xi$.
\epr


\subsection{Proof of Theorem \ref{t.main}}
\label{ss.ProofMain}

To finish the proof of the theorem we successively apply the Iterative Lemma   \ref{l.itHam} as follows:
We   first apply the symplectic transformation from the Iterative Lemma \ref{l.itHam} 
three times to introduce $(w_3,z_3)\mapsto(w_2,z_2)\mapsto (w_1,z_1)\mapsto (w_0,z_0)$
taking 
 \[
 \xi=\xi_0=\xi_1=\xi_2 =\frac16 \min\,(\nu_0-  \nu,\sigma_0- \sigma).\]
We choose $\delta_0>0$ and $\epsilon>0$ sufficiently small to satisfy  \eqref{e.deltaCond} and the conditions \eqref{e.ItLeHam-conds} and \eqref{e.K+} of 
   Lemma \ref{l.itHam} for the above choice of $\xi_0$.  Applying this lemma  once we obtain  $\delta_1 =\kO( \epsilon)$.
For all successive iterations we use the following bound for $C_{F_n}''[\nu_n, K_n\delta_n]$ from \eqref{e.CF''}
where $K_n$, $\delta_n$ etc.~denote the constants
of Lemma \ref{l.itHam} after $n$ iterations:
We set
\begin{equation}
\label{e.F''}
 C_{F_n}''[\nu_n, K_n\delta_n]=  \frac{2 C_{F_n}[\nu_n,K_n \delta_n+\kappa]}{\kappa^2},
\end{equation} 
by applying  the Cauchy estimate \eqref{e.cest}.  
 Here $\kappa> 0$, $K_n\delta_n+\kappa \leq\sigma_n-\xi_n/2$ and  $ C_F[\nu_n,K_n \delta_n+\kappa]$ is such that
$\Wert F_n\Wert_{\nu_n, K_n\delta_n+\kappa} \leq C_{F_n}[\nu_n,K_n \delta_n+\kappa]$.
 We let $\sigma_n-\xi_n \geq \sigma$ for all $n\leq N$ with $N\in\N$ to be determined later so that we can choose $\kappa = \sigma$ (noting that $K_n \delta_n \leq \xi_n/2$ by \eqref{e.ItLeHam-conds}).
We then use the  condition   
\begin{equation}\label{e.delta-ham}
\delta_n<   \frac{\kappa^2  }{ K_n^{2} C_{F_n}[\nu_n,K_n \delta_n+\kappa]}
\end{equation}
 instead of \eqref{e.deltaCond} in the following.

Since $\delta_1= \kO(\epsilon)$ we can satisfy \eqref{e.delta-ham}  and the other conditions \eqref{e.ItLeHam-conds} and \eqref{e.K+} of 
   Lemma \ref{l.itHam} for sufficiently small $\epsilon$
and therefore apply Lemma \ref{l.itHam} twice to obtain $\delta_2 =\kO( \epsilon^2)$ and $\delta_3 =\kO( \epsilon^3)$.
 We can then  ensure  
 that $\delta_3$ is  small enough so that
for $n\geq 3$ the conditions of Lemma \ref{l.itHam},  \eqref{e.delta-ham},  \eqref{e.ItLeHam-conds} and \eqref{e.K+}, 
 are satisfied for the choice $\xi_n = \kO(\epsilon)$.
We now  apply  Lemma  \ref{l.itHam} successively starting from $(\kV+\i\nu_3)\times (\kS+\i\sigma_3)$ with 
\begin{align}\label{e.nu2sigma2}
\nu_3-\nu=
\sigma_3-\sigma =\frac12 \min(\nu_0-\nu, \sigma_0-\sigma).
\end{align}
  Note that 
  \[
  \Vert h_3-h_0\Vert_{\nu_3},\,\Vert a_3-a_0\Wert_{\nu_3},
 \,\Vert f_3-f_0\Vert_{\nu_3,\sigma_3 }, \,\Vert F_3-F_0\Wert_{\nu_3,\sigma_3-\xi_0/2 } =\kO(\delta_0). 
 \]
 We then apply the transformations 
\begin{align*}
\Psi_n:&(\kV+\i\nu_{n+1})\times (\kS+\i\sigma_{n+1})\rightarrow (\kV+\i(\nu_n-\xi_n/2))\times (\kS+\i(\sigma_n-\xi_n/2)),\\
&(w_{n+1},z_{n+1})\quad \mapsto \quad (w_n,z_n),
\end{align*}
iteratively, with 
\begin{equation}
\label{e.xin}
\xi_n = 2 c_* \epsilon \geq  2\epsilon\max(4 C_\kS, c_n),
\end{equation}
with 
$\nu_n=\nu_3-\sum_{i=3}^{n-1} \xi_i\ge {\nu}$, $\sigma_n = \sigma_3-\sum_{i=3}^{n-1} \xi_i\ge {\sigma}$ for $3\leq n\leq N$ with $N\in \N$
and   $c_*$ to be determined later.  Here we choose $\epsilon>0$ small enough such that $\xi_n\leq 1$.
 This choice of $\xi_n$ ensures  that 
\begin{align*}
 \delta_{n+1}\le \frac{1}{2}\delta_{n},
 \end{align*}
 cf. \eqref{e.deltaHam-est},
 and that  
 \begin{align*}
  K_{n+1} - K_n, C_{f_{n+1}}- C_{f_n}, C_{a_{n+1}}-  C_{a_n}, C_{h_{n+1}}-  C_{h_n}, C'_{h_{n+1}}- C'_{h_n}   \leq \frac{c_n\delta_n}{\xi_n}\le \frac{\delta_n}{\epsilon},
\end{align*}
and
\begin{align*}
C_{F_{n+1}}-  C_{F_n} \leq \frac{c_n\delta_n}{\xi_n^2}\leq \frac{\delta_n}{\epsilon\xi_n} \leq \frac{\delta_n}{8C_\kS \epsilon^2 },
\end{align*}
cf.~\eqref{e.ham-it-ests}, \eqref{e.hamK+}, \eqref{e.xin}, where we take $\epsilon$ small enough such that $\xi_n < \min(\nu_n-\nu,\sigma_n-\sigma)$
for $n\leq N$, with $N$ to be determined later.
 Then
\begin{align*}
\delta_{n+1}\leq 2^{-n+2} \delta_3,
\end{align*}
where $\delta_3 = \kO(\epsilon^3)$.  This proves that  the constants from the Iterative Lemma \ref{l.itHam} are bounded uniformly with respect to $3\leq n\le N$ with
\begin{equation}\label{e.boundKnEtc}
 K_{n} - K_3,  C_{F_{n}}-  C_{F_3}, C_{a_{n}}-  C_{a_3}, C_{h_{n}}-  C_{h_3}, C'_{h_{n}}- C'_{h_3}= \kO(\epsilon).
 \end{equation}
  Since the constant $c$ from Lemma \ref{l.itHam} is increasing and continuous in $K$, $C_F$, $C_a$, $C_h$ and $C'_h$ there is $  c_*$ such that  
$ c_n\leq   c_*$  for all $n\leq N$.
We choose $c_*\geq 4 C_\kS$ and set $\xi_n = 2 c_* \epsilon$, see \eqref{e.xin}.
Because of  the inequalities
\begin{align}
 {\nu}&\le \nu_3 -  2 c_* \epsilon  (N-2) \le  \nu_{N+1}= \nu_3 - \sum_{i=3}^{N}\xi_i, \notag \\
 {\sigma}&\le \sigma_3 -   2 c_* \epsilon  (N-2)\le \sigma_{N+1}= \sigma_3 - \sum_{i=3}^{N}\xi_i,
\label{e.nu-nu2}
\end{align}
noting that we want to define the transformed Hamiltonian on $\kV+\i\nu \times \kS+\i\sigma$ and using \eqref{e.nu2sigma2},  we
take $N$ to be 
\begin{align}\label{e.N}
N=\left\lceil\frac{M}{4  c_* \epsilon}\right\rceil,\quad M=\min\,(\nu_0-  \nu,\sigma_0-  \sigma).
\end{align}
Here we  denote by $\lceil x\rceil$ the smallest integer $\geq x$. This completes the proof of Theorem \ref{t.main}.

\subsection{Remarks on the proof of  Theorem \ref{t.main}}

The following    remark shows that
we can construct the  slow manifold such that it contains a given equilibrium  of the Hamiltonian slow-fast system \eqref{e.HamPDE}.

\begin{remark}
Assume (H0-H3) and let $\delta_0>0$ and $\epsilon>0$ be sufficiently small.  In addition assume  that there exists a locally unique equilibrium  of \eqref{e.HamPDE} which in the $(w_0,z_0)$-coordinates takes the form {$(w^e,0)\in \kV\times \kS$}. Then  the equilibrium 
  $ (w^e,0) $ is a fixed point of $\Psi$.
\end{remark}
\bpr
Let $(w^e_{+},z^e_{+})$ be such that
 $ (w^e,0) = \Psi(w^e_{+},z^e_{+})$ where $\Psi$ is the symplectic transformation 
  from  \eqref{e.G-psi}, \eqref{e.alphaBeta}  used in the Iterative Lemma \ref{l.itHam}.
 First note that $z^e = 0$ implies that $\zeta(w^e) = 0$.
 Moreover from the definition of $\Psi$  we have
\begin{align*}
z^e_{+} &= z^e -\zeta(u^e,v^e_{+}) = -\zeta(u^e,v^e_{+}),\\
u^e_{+} &= u^e-\epsilon \langle J_\kZ^{-1}\partial_{v_+}\zeta(u, v_+), (x,y_+)\rangle,\\
v^e_{+} &= v^e+\epsilon  \langle J_\kZ^{-1}\partial_{u}\zeta(u, v_+), (x,y_+)\rangle.
\end{align*}
Insertion then proves the result.
\epr

Also note that we can assume   that the equilibrium $(w^e, z^e)$ in the above remark  is at $z^e=0$   without loss of generality,  by introducing the affine {symplectic} 
transformation $(w_0,z_0)\mapsto (w_0,z_0-z^e)$. It is important to 
start our iteration from 
$z^e=0$ - we can then ensure that  the  slow manifolds,  that we defined
iteratively in the proof of Theorem \ref{t.main}, contain this equilibrium. Obviously we could also transform $w^e=0$ but
 this is not necessary. 

\begin{remark}
\label{r.Lu} \rm
Lu \cite{Lu} uses our method for obtaining  a symplectic slow manifold as presented in an earlier
preprint version of this paper to study breathers in a semilinear wave equation
\begin{equation}
\label{e.LuSWE}
u_{tt} = u_{xx} -u + f(u)
\end{equation}
where $f(u)$ is odd, holomorphic and  $f'(0)=0$. He studies \eqref{e.LuSWE} on $2\pi/\omega$ odd periodic functions
where the lowest Fourier mode $\sin x$ corresponds to the slow dynamics (after rescaling $x \to x/\omega$).
He transforms $v=u_t$ such that the transformed system becomes well-posed on the subspace $\kZ$ of odd functions in 
$\kH_1\times \kH_1$. After several other transformations the resulting fast system (in the fast time) takes the form
\begin{equation}\label{e.LuPDE}
\dot z = J_\kZ L z  + \epsilon^2 B(w,z).
\end{equation}
So compared to \eqref{e.HamPDE} the nonlinearity is of order $\epsilon^2$. 
Instead of solving $\dot z=0$ in Lemma \ref{l.analimpl}, Lu just solves $Lz = -\epsilon^2 r(w)$, 
where $B(w,z) = J_\kZ r(w) + \kO(z)$,  for
$\hat \zeta(w) = -\epsilon^2 L^{-1}r(w)$ and defines the symplectic transformation $\Psi$ from Lemma \ref{l.sglem}
used in the Iterative Lemma \ref{l.itHam}, with $\hat\zeta(w)$ instead of $\zeta(w)$. For the special case \eqref{e.LuPDE} the error $\delta$ of his construction  of the slow manifold   still shrinks by an order of $\epsilon$ in each step, and this
 simplifies the proof considerably. 
\end{remark}

 
\section{An invariant two-dimensional slow manifold} \label{s.psm}
In this section we prove the existence of a two dimensional normally elliptic slow manifold with exponentially 
small gaps under the following assumptions:
Consider again a real analytic slow-fast Hamiltonian system with Hamiltonian 
$H_0(w_0,z_0)$, but now with a single
slow degree of freedom, which in addition to (H0-H3) satisfies the following assumptions:

 \begin{itemize}
 \item[(I1)]   $d_{\kW}=1$, and  $z_0=0$ is invariant at $\epsilon=0$ for $w_0\in \kV+\i\nu_0$. 
  Moreover $\{z_0=0,\epsilon=0\}$ is filled with a family of     periodic orbits parametrized by energy $E \in (E_1, E_2)+ \i e_0$.
  Their frequency $\omega^E$ as a function of energy satisfies $\omega^E \neq $ and
   $\frac{\partial \omega^E}{\partial E} \neq 0$ for $E \in (E_1, E_2)+ \i e_0$, $e_0>0$. 
\item[(I2)]
$\dim \kZ= 2d_\kZ<\infty$,
$J_\kZ$ is standard, 
     and    $A_0(w_0) $ is of the form 
\begin{align}
 A_0(w_0) = L+a_0(w_0) = L + \epsilon M_0(w_0),\label{e.Aform}
\end{align}
suppressing the $\epsilon$-dependency in $M_0(w_0)$, with
\begin{align*}
 L=\text{diag}\,(\omega_1,\ldots,\omega_{d_\kZ},\omega_1,\ldots,\omega_{d_\kZ}).
 \end{align*}
Moreover, $\delta_0 = \kO(\epsilon)$.
\item[(I3)] 
We have $\omega_{i}\ne 0$ for all $i$ and the following non-resonance condition holds:
\begin{align}
\forall \ell\ne m \quad \omega_{\ell}\ne \omega_m. \label{e.nonres}
\end{align}
\end{itemize}

  Then by Theorem \ref{t.main}
  there exists a symplectic map $(w,z)\mapsto (w_0,z_0)$ that transforms the Hamiltonian into
\begin{align}
 H(w,z) = h(w)+r(w)\cdot z+\frac12 A(w)z\cdot z +{f}(w,z),\label{e.hampsm0}
\end{align}
defined on "$(\kV+\i \nu) \times (\kS+\i \sigma)$, with
 $\|r\|=\kO(\e^{-C/\epsilon})$ provided $\epsilon$ is sufficiently small. 
 Here, as before, $f = \kO(\|z\|^3)$, and from  $\delta_0 = \kO(\epsilon)$, we conclude
 that  $A(w) = L + \epsilon M(w)$ has the same form as $A_0(w_0)$ so that (I2) and (I3) hold 
for $A(w)$ too. Moreover (I1) holds for the $\dot w$ equation at $\epsilon=0$ 
for $E \in (E_1, E_2) + \i e$ with $e_0\geq e>0$.

Note that if both \eqref{e.Aform} and \eqref{e.nonres} are not satisfied 
then ``Takens chaos'' can occur, see \cite{Takens}.
 In this section we consider \eqref{e.HamPDE} on the slow time scale $\tau =  \epsilon t$.

In words, the result of this section is then the following: A periodic orbit for  the $\dot w$ equation at $\epsilon=0$ can be continued into the full system with Hamiltonian \eqref{e.hampsm0} provided that there is no resonance with the fast system. If there is  a resonance, then   this only excludes exponentially small bands in the $w$-plane of periodic orbits. This result can be viewed as an extension of a result of Gelfreich and Lerman in \cite{geller2} to several fast variables. 

\begin{remark}\rm
The setting considered in this section applies to the LK model in \cite[(2.7)-(2.10)]{van1} and the generalized conservative versions \cite[(6.1)-(6.2)]{van1} with $s\in \R^2$ and $\kL(s)$, as defined in \cite{van1}, independent of $s$, and the main result (Theorem \ref{t.sm}, below) therefore applies to these examples. 
\end{remark}

Before stating the result (Theorem \ref{t.sm} below), we perform a sequence of simplifications serving to bring the system into a form appropriate for application of the contraction mapping theorem. It is important to note that we are not connecting with $\epsilon=0$. Instead we are introducing an artificial perturbation parameter $\mu$.

First we transform $h=h(w)$  into action-angle variables $(I,\phi)$.
We have
\begin{lemma}\label{l.ActionAngle} Assume (H0-3) and (I1-I3).
 Let $(I,\phi)\mapsto w$ be the symplectic change of coordinates which transforms
 $h(w)$ to   $h(w)=\breve h(I)$. This transformation is analytic from $w \in \kV + \i \nu $ 
 to $ (\phi, I) \in ([0,2\pi]+ \i \psi) \times ((I_1, I_2) + \i \iota)  $
 for some $\psi, \iota>0$.  Here $[0,2\pi]+ \i \psi$  is
a complex neighbourhood of length $\psi$ around $[0,2\pi]$  and
$(I_1, I_2) + \i \iota$  a complex neighbourhood of length $\iota$ around $(I_1, I_2) \subseteq \R^+$.
This map transforms \eqref{e.hampsm0} into
 \begin{align}
H(w, z) = \breve H(\phi, I, z) =\breve  h(I)+\langle \breve r(\phi,I), z\rangle +\frac12 \langle \breve A(\phi,I)z, z\rangle +\breve f(\phi,I,z),\label{e.Hpsm}
\end{align}
where $\breve h(I) = h(w)$, $\breve r(\phi,I)=r(w)$, $\breve A(\phi,I)=L+\epsilon \breve M(\phi, I)$, $ \breve M(\phi,I) = M(w)$, and $\breve f(\phi,I,z)=f(w,z)$. 
 \end{lemma}
\begin{proof}
 The system with Hamiltonian $h=h(w)$
is integrable since it is a one-degree of freedom system. This transformation does not depend 
upon the fast variables and can therefore directly be lifted to the full space.
\end{proof}

Next we reduce to an  energy level:
Since $ \omega^E\neq 0$ by (I1) we have $\omega(I) := \partial_I\breve h(I) \ne 0$, and so we can solve
 the equation $\breve H(\phi,I,z)=E$  for $$I=I^E(\phi,z),$$ when $z \in \kB_\sigma^{\kZ^\C}(0)$ 
 by potentially decreasing $\sigma>0$, 
 and we may introduce the angle $\phi$ as new time.
 \begin{lemma} Under the above assumptions 
 $z=z(\phi)$  solves the following non-autonomous Hamiltonian system of
 equations:
\begin{align}
 \epsilon z'(\phi) &= -J_{\kZ}\nabla_z I^E(\phi,z).\label{e.zprime}
\end{align}
  \end{lemma}
\begin{proof}
By definition 
\begin{align*}
 \epsilon \dot z  =  \epsilon \frac{\d z}{\d\phi} \partial_I \breve H = J_{\kZ} \nabla_z \breve H.
\end{align*}
 Next, we differentiate 
 \begin{align}
 \breve H(\phi,I^E(\phi,z),z)=E,\label{e.KE}
 \end{align}
  with respect to $z$ to obtain
 \begin{align}
  \partial_I \breve H \nabla_z I^E =- \nabla_z  \breve H.\label{e.KEz}
 \end{align}
This completes the result. 
\end{proof}

\begin{lemma}\label{l.IE}
  Under the above assumption $I^E = I^E(\phi,z)$ takes the following form
  \begin{align}
   I^E(\phi,z) &=\breve h^{-1}(E)-\frac12  \langle A^E(\phi) z, z\rangle-
   \langle r^E(\phi), z\rangle - f^E(\phi,z),\label{e.KE2}
  \end{align}
    with 
   \[
  r^E(\phi) =  \breve r(\phi, \breve h^{-1}(E))/\omega^E,
  \quad   f^E(\phi,z) = \kO(\| z\|^3)
  \]
    and
  \[
   A^E(\phi):= \breve A(\phi,\breve h^{-1}(E))/\omega^E   +  \kO(e^{-C/\epsilon}) =L^E + \epsilon M^E(\phi),
    \]
      where
      \begin{align}
       \omega^E = \omega(\breve h^{-1}(E)) \label{e.OmegaE}
      \end{align}
and
  \[ 
   L^E =  L/\omega^E, \quad M^E = \breve M(\phi, \breve h^{-1}(E))/\omega^E   +  \kO(e^{-C/\epsilon}).
   \]
   Here $I^E$, $f^E$, $M^E$ and $L^E$ are analytic in $\phi \in [0,2\pi]+\i\psi$, $E \in (E_1,E_2)+\i e$ and $z \in \kB_\sigma^{\kZ^\C}(0)$.
   Finally
   \begin{equation}\label{e.rE}
   \| {r}^E\|_\psi := \max_{\phi \in [0,2\pi] + \i \psi}\| r^E(\phi)\| = \kO(e^{-C/\epsilon})
   \end{equation}
  uniformly for $E \in (E_1,E_2)+\i e$.
 \end{lemma}
\begin{proof}
 Equation \eqref{e.KE} with $z=0$ gives
 \begin{align*}
  I^E(\phi,0) =\breve h^{-1}(E),
 \end{align*}
 cf. \eqref{e.Hpsm}.    Setting $z=0$  in \eqref{e.KEz} then gives
\begin{align}
 \nabla_z I^E\vert_{z=0} =- \breve r(\phi, \breve h^{-1}(E))/\omega^E=-r^E(\phi)  =  \kO(\e^{-C/\epsilon}).
 \label{e.nablaIE}
\end{align}
If we differentiate \eqref{e.KE} again with respect to $z$   
 we get \begin{align*}
0 & =  \partial_z^2\breve H\vert_{z=0}  +  \partial_I^2 \breve H\vert_{z=0} (\partial_z I^E\vert_{z=0})^2
  +2( \partial_I\partial_z \breve H\vert_{z=0} \partial_z I^E\vert_{z=0})_{\rm sym} \\
&  \quad  + \partial_I  \breve H\vert_{z=0} \partial_z^2 I^E\vert_{z=0},
 \end{align*}
and so we find that
\begin{align*}
 \partial_z \nabla_z I^E\vert_{z=0}  & = -( \breve A(\phi,\breve h^{-1}(E)) +  \partial^2_I \breve h(I)   \nabla_z I^E\vert_{z=0} \partial_z I^E\vert_{z=0} )/ \omega^E,\\
   & \quad - (  \partial_I \breve r(\phi,I) \partial_z I^E\vert_{z=0} +   (\partial_I \breve r(\phi,I) \partial_z I^E\vert_{z=0})^T )/\omega^E
\end{align*}
where $I = \breve h^{-1}(E)$. Using \eqref{e.nablaIE} this gives the result.
\end{proof}

Now we fix $\epsilon$ small and introduce $\|r^E\|_\psi \leq \mu^2=\kO(\e^{-C/\epsilon})$ as a measure of the remainder in \eqref{e.KE2}.
Setting $z=\mu\hat z$ then transforms  \eqref{e.zprime}  into
 \begin{align}
  \epsilon \hat z'= J_{\kZ}  ( A^E(\phi) \hat z + \hat r^E(\phi) + \hat F^E(\phi, \hat z)) \label{e.barz}
  \end{align}
  where
  \[
   \hat r^E(\phi)  :=  r^E(\phi)/\mu, \quad \hat F^E(\phi, \hat z) :=  F^E(\phi, \mu \hat z)/\mu = 
  \nabla_z f^E(\phi, \mu \hat z)/\mu,
  \]
 and $F^E = \nabla_z f^E$. Choosing $\epsilon>0$ small enough such that $\mu< 1/2$ we see that $\hat F^E(\phi, \cdot)$ is analytic
 on $\kB_{2\sigma}^\kZ(0)$. Then due  to \eqref{e.rE}  and  due to the fact that $ F^E(\phi,   z) = \kO(\|z\|^2)$
 we obtain 
  \begin{align}\label{e.hatFE}
\|\hat r^E\|_\psi & := \sup_{\substack{ \phi \in [0,2\pi] + \i \psi}} \|\hat r^E(\phi)\| = \kO(\mu),\quad
\|\hat F^E\|_{\psi, 2\sigma}  := \sup_{\substack{ \phi \in [0,2\pi]+\i\psi,\\ \|z\|\leq   2\sigma  }}\|\hat F^E(\phi, \hat z) \| = \kO(\mu)
   \end{align}
   uniformly in $E \in (E_1, E_2)$.
  In the notation $\|\hat F^E\|_{\psi,2\sigma}$, and in what follows we adapt the definition
 from (H3), with $\kS=\{0\}$.

Let $$\Pi^{\mu, E}:\{\phi=0\}\rightarrow \{\phi=2\pi\}$$ be the stroboscopic mapping obtained from \eqref{e.barz}.
 It is symplectic since the system is Hamiltonian. 
Note that $\Pi^{0, E}(0)=0$ due to (I1). The persistence of this fixed point 
  for $\mu \ne 0$ provides the persistence of the periodic orbits, which we have parametrized by $E$.
To study this mapping we consider the monodromy matrix $\Psi^E(2\pi, 0)$ associated with the linear problem (obtained from \eqref{e.barz} by setting $\mu=0$)
\begin{align}
 \epsilon \hat z'&=   J_{\kZ}A^E(\phi) \hat z=   J_{\kZ} \left( L^E+\epsilon  M^E(\phi)\right)\hat z.\label{e.newform}
\end{align}
The eigenvalues of $\Psi^E(2\pi, 0)$, $\lambda_1^E,\ldots,\lambda_{2 {d_\kZ}}^E$, are the characteristic multipliers of $\Pi_{\mu,E}$ at $z=0$. 
Exploiting the form of $A^E(\phi) = L^E+\epsilon M^E(\phi)$  and the non-resonance condition \eqref{e.nonres} we can approximate those very accurately using the following lemma:

\begin{lemma}\label{l.T}
Under the above assumptions 
 for all  sufficiently small $\epsilon>0$ there exist $\tilde C>0$ and  a  linear change of variables $\hat z\mapsto\tilde z =\hat z + \epsilon  T^E(\phi) \hat z$, which is $2\pi$-periodic in $\phi$, 
  and  is  analytic in $(\phi, E) \in [0,2\pi]\times (E_1, E_2)$  which transforms \eqref{e.newform} into
\begin{align}
\epsilon \tilde z'_j  &=  (\omega_j+\epsilon b_j^E(\phi))\tilde z_{j+d_\kZ}/\omega^E +  (R^E(\phi) z)_j \notag,\\
\epsilon\tilde z'_{j+d_\kZ} &=  -(\omega_j+\epsilon b_j^E(\phi)){\tilde z_j}/\omega^E+ (R^E(\phi) z)_{d_\kZ+j}.\label{e.redq}
\end{align}
Here  $b_j^E(\phi)$,
 $j=1,\ldots d_{\kZ}$, are  scalar analytic functions  and $R^E(\phi)$ 
    is a $(2 d_\kZ, 2d_\kZ)$ symmetric matrix which is analytic  in $(\phi,E)  \in [0,2\pi]\times (E_1, E_2)$ and satisfies
    \begin{equation}\label{e.REphi}
\|R^E\|:= \max_{ \phi \in [0,2\pi]}   \| R^E(\phi) \| = \kO(\e^{-\hat C/\epsilon})
    \end{equation}
    uniformly in $E \in (E_1, E_2)$.
\end{lemma}

\begin{proof}
Consider the system
\begin{equation}\label{e.ItEq2dSM}
\dot z = J_\kZ (L(\phi ) + \epsilon M(\phi))z
\end{equation}
where 
\[
L(\phi) = {\rm diag}(\lambda_1(\phi), \ldots,\lambda_{d_\kZ}(\phi),\lambda_1(\phi), \ldots, \lambda_{d_\kZ}(\phi)) 
\]
 is diagonal,  $M(\phi)$ is symmetric 
 and $L(\phi)$ and $M(\phi)$ are analytic in $\phi \in [0,2\pi]+\i\psi$
 and also depend on $\epsilon$ and $E$.
Note that \eqref{e.newform} is of the form \eqref{e.ItEq2dSM},
see Lemma \ref{l.IE},
 and we will use \eqref{e.ItEq2dSM} to set up an iterative lemma. 

For $\epsilon>0$ small enough let   $\id  + \epsilon T(\phi)$
be the  linear coordinate transformation such that
 \[
( \id  + \epsilon T(\phi))J_\kZ (L(\phi) + \epsilon M(\phi)) ( \id  + \epsilon T(\phi))^{-1}
= J_\kZ L_+(\phi)  
\]
where  
$L_+(\phi) $ has the same form as $L(\phi)$.
Then  $z_+ :=( \id  + \epsilon T(\phi))z$ defines a symplectic change of coordinates.
Moreover there are matrix valued functions $G$  and $F$ 
with $G(L,0) = 0$, $F(L,0) = 0$ such that
\[
L_+ = L + G(L,\epsilon M), \quad T  = F(L,\epsilon M).
\]
Both $G$ and $F$ are analytic as functions on $\kB_1 \times \kB_2$. Here $ \kB_1$ is ball
 of radius $r_1>0$ around $L_0:=L^E$ from \eqref{e.newform} in the $d_\kZ$ dimensional space
  of complex diagonal $(n,n)$ matrices, where $n=2d_\kZ$, 
 $\kB_2$ is a ball of radius $r_2>0$ around $0$ in the space   of  symmetric $(n,n)$ matrices,  and $r_1$, $r_2$  are small enough such that
$ J_\kZ (L+ S)$ has disjoint eigenvalues for $L  \in \kB_1$, $S \in \kB_2$.
  Therefore there are some smooth  non-decreasing
functions $f$ and $g$
mapping a neighbourhood of $0$ in $\R^2$ into $\R$ such that
\begin{align*}
\| F(L,M)\|_\psi  & \leq \| M\|_\psi  f(\|L-L_0\|_\psi, \|M\|_\psi),
\\
 \| G(L,M)\|_\psi & \leq \| M\|_\psi  g(\|L-L_0\|_\psi, \|M\|_\psi).
\end{align*}
Now $z_+ = ( \id  + \epsilon T(\phi)) z$  satisfies
\begin{align*}
\epsilon z'_+ &= ( \id  + \epsilon T(\phi))  \epsilon  z'
+ \epsilon^2 T'(\phi)  z \\
& = 
( \id  + \epsilon T(\phi))      J_{\kZ} \left( L(\phi)+\epsilon  M(\phi)\right) z  +   \epsilon^2 T'(\phi)  z\\
& =  J_\kZ L_+(\phi)  z_+  +  \epsilon^2 T'(\phi) ( \id  + \epsilon T(\phi))^{-1} z_+
\\
& =  J_\kZ L_+(\phi)  z_+  +  \epsilon M_+(\phi)  z_+
\end{align*}
where 
\[
  M_+(\phi) =   \epsilon T'(\phi) ( \id  + \epsilon T(\phi))^{-1}.
\]
Using Cauchy's estimate we get, with $B = L-L_0$,
\[
\| M_+\|_{\psi-\xi}  \leq \epsilon \frac{ \|T\|_\psi}{\xi(1 - \epsilon \|T\|_\psi) }
\leq  \frac{ \epsilon \|M\|_\psi  f(\|B\|_\psi, \epsilon\|M\|_\psi)}{\xi(1- \epsilon  \|M\|_\psi  f(\|B\|_\psi,\epsilon \|M\|_\psi))} \leq \|M\|_{\psi}/2
\]
if
\[
\xi = 2 \epsilon c, \quad\mbox{where}\quad
c \geq  \sC(\|B\|_\psi, \|M\|_\psi):= \max\left(1, \frac{   f(\|B\|_\psi, \epsilon\|M\|_\psi)}{1- \epsilon  \|M\|_\psi  f(\|B\|_\psi, \epsilon\|M\|_\psi)}\right).
\]
When iterating this procedure we need to ensure that $c_n = \sC(\|B_n\|_{\psi_n},\epsilon \|M_n\|_{\psi_n})$  is bounded   independent of $n$.
For this we first show that 
$C_{B_n} : = \|B_n\|_{\psi_n}$ where $B_n = L_n-L_0$,  is sufficiently small and
 $C_{M_n} : = \|M_n\|_{\psi_n}$ bounded,  so that $L_n \in \kB_1$ and $\epsilon M_n\in \kB_2$ for  all $n\leq N$ (with $N$ to be determined later)  provided $\epsilon>0$ is sufficiently small. Here $L_0=L^E$ and $M_0=M^E$  are as in \eqref{e.newform} and
$\psi_n = \psi - \sum_{j=1}^{n-1}\xi_j$.
This follows from the following estimates:
\begin{align*}
C_{M_{n}} & \leq C_{M_{n-1}}/2 \leq 2^{-n} C_{M_0},\\
C_{B_{n}}  & \leq \sum_{j=1}^n \|L_{j} - L_{j-1}\|_{\psi_j}
\leq \epsilon \sum_{j=0}^{n-1} C_{M_j} g_j \leq 2 \epsilon  C_{M_0} \max_{j=0,\ldots n-1} g_j,
\end{align*}
where
$g_j:= g(C_{B_j},\epsilon C_{M_j})$. This shows that $c_n$ is bounded for all $n$ if $\epsilon$ is small enough.
Let $c_* = \sup_{n=0,\ldots N-1} c_n$,  
define  $\xi_n = 2c_* \epsilon$ and let $N$ be the largest number such that 
\[
\psi - N \xi= \psi-2Nc_*\epsilon>0,\]
 so
$N := \lfloor  \frac{\psi}{2 c_*\epsilon}\rfloor$, the largest integer $\leq \frac{\psi}{2 c_*\epsilon}$.   Then the norm of $R^E(\phi):= \epsilon M_{N}(\phi)$ is bounded by $\epsilon C_{M_N} \leq 
\epsilon 2^{- \lfloor\frac{\psi}{2 c_*\epsilon}\rfloor} C_{M_0}$. We then set $T^E(\phi) = (T_N \circ\ldots \circ T_1)(\phi)$ and
$b_j^E =\omega^E (L_N-L_0)_j/\epsilon$, $j=1,\ldots  d_\kZ$.
\end{proof}
 
We can solve \eqref{e.redq} to obtain
\begin{equation}\label{e.Psi}
 \tilde z(2\pi) = \tilde\Psi^E(2\pi,0) \tilde z(0) =  \tilde\Phi^E(2\pi,0) \tilde z(0) + \kO(\epsilon^{-1}\e^{-\tilde C/\epsilon})\tilde z(0),
 \end{equation}
 where
\begin{align*}
 \tilde \Phi^E(2\pi,0) = \begin{pmatrix}
                          \cos\left(\alpha^E/\epsilon\right) & \sin\left(\alpha^E/\epsilon\right)\\
                          -\sin\left(\alpha^E/\epsilon\right) & \cos\left(\alpha^E/\epsilon\right)
                         \end{pmatrix}
\end{align*}
and 
\begin{align}
 \alpha^E= \text{diag}\left(\alpha_1^E,\ldots,\alpha_\ell^E,\ldots,\alpha_{d_{\kZ}}^E\right):= \int_0^{2\pi} \text{diag}(\omega+\epsilon b^E(s))\d s/\omega^E.\label{e.alpha}
\end{align}
 The eigenvalues of $\tilde\Psi_E(2\pi, 0)$ are therefore
\begin{align}
 \lambda_\ell^E &= \exp\left(\i\epsilon^{-1} \alpha_\ell^E\right) + \kO(\epsilon^{-1}\e^{-\tilde C/\epsilon}),\quad \lambda_{\ell+d_{\kZ}}^E = \bar\lambda_\ell^E, \label{e.eigl} 
 \end{align}
for $\ell=1,\ldots, d_\kZ$.
We write $\Pi^{\mu, E}( \hat z) = \tilde \Pi^{\mu, E}(\tilde z) $ in the coordinates $\hat z \to \tilde z= \hat z + \epsilon T^E(\phi) \hat z$ from Lemma \ref{l.T} using variations of constants in \eqref{e.barz}:
\[
  \tilde\Pi^{\mu,E}( \tilde z)   = \tilde\Phi^E(2\pi,0) \tilde z +\tilde \rho^E(   \tilde z),
  \]
  where 
  \[
  \tilde \rho^E(  \tilde z_0) =
   \int_0^{2\pi} \tilde\Phi^E(2\pi, s) ( J_\kZ  \tilde r^E(s)
 +\epsilon^{-1} R^E(s) \tilde z(s) + J_\kZ\tilde F^E(s, \tilde z(s) ))\d s
  \]
 with $\tilde z(0) = \tilde z_0$ and
 \[
  \tilde r^E(\phi) = (\id + \epsilon T^E(\phi)) \hat r^E(\phi)/\epsilon, \quad \tilde F^E(\phi, \tilde z) =  (\id + \epsilon T^E(\phi)) \hat F^E(\phi,  (\id + \epsilon T^E(\phi))^{-1} \tilde z)/\epsilon.
 \]
 Choosing $\epsilon>0$ small enough such that $\epsilon\|T^E\|\leq 1/2$
 we obtain that  $\tilde F^E(\phi,\cdot)$ is analytic on $\kB_\sigma^\kZ(0)$.
 Due to \eqref{e.hatFE}  we have 
 \begin{align*}
  \|\tilde F^E\|_\sigma: = \sup_{\substack{\|z\|\leq \sigma\\ \phi \in [0,2\pi]}} \|\tilde F^E(\phi,z)\| =\kO(\epsilon^{-1}\mu),
  \quad
    \| \tilde r^E\| = \sup_{  \phi \in [0,2\pi] } \| \tilde r^E(\phi)\|= \kO(\epsilon^{-1}\mu), 
 \end{align*}
uniformly for $E \in (E_1, E_2)$,
 and therefore also 
 \begin{equation}\label{e.tildeRhoE}
  \|\tilde\rho^E \|_{\tilde\sigma}  = \kO(\epsilon^{-1}\mu)
 \end{equation}
uniformly in $E \in (E_1, E_2)$.
 Here we redefine $\mu^2= \e^{-C/\epsilon}$, with
 $C \leq  2\tilde C$ and set $0<\tilde\sigma<\sigma$ such that $\sigma-\tilde\sigma>\kO(\epsilon^{-1}\mu)$.
Whenever $(\id-\tilde\Phi^E(2\pi,0))^{-1}$ exists the fixed points  of $\tilde\Pi^{\mu,E}$ satisfy
 \begin{align}
    \tilde z =  \pi^{\mu,E}(\tilde z):=  (\id- \tilde\Phi^E(2\pi,0))^{-1}\tilde\rho^E(\tilde z).\label{e.fpEqn}
 \end{align}

  Then we have:
 \begin{theorem}\label{t.sm}
 Under assumptions (H0-H3), (I1-I3) for any $\epsilon>0$ sufficiently small there is a two-dimensional 
 manifold  $\kM_\epsilon$ of non-degenerate periodic orbits parametrized by energy $E \in (E_1, E_2) \setminus I$
 where $I$ is a union of  $\kO(\epsilon^{-1})$-many intervals, and the measure $|I|$ of $I$ is exponentially small:
 $|I| = \kO(\e^{-c/\epsilon})$ for some $c>0$.
\end{theorem}
\begin{proof}

 Let $E=E_0$ be a bifurcation value, i.e., $\tilde\Phi^{E_0}(2\pi, 0)$ has  at least two eigenvalues which are $1$ (they come in pairs). Then, for $\epsilon>0$ sufficiently small the following condition is satisfied:
\begin{align}
  \partial_E \alpha_\ell^{E}|_{E=E_0} \ne 0,
 \label{e.degcond}
\end{align}
for all $\ell$ with $\lambda_\ell^{E_0}=1$. 
Note that \eqref{e.degcond} follows from 
\begin{align*}
\partial_E \omega^E|_{E=E_0}\ne 0,
\end{align*}
cf. \eqref{e.OmegaE}, 
for $\epsilon$ small which is guaranteed by (I1). 
Let $\P_\ell$ be the  projection to the space spanned by $e_\ell$ and $e_{\ell+d_\kZ}$. Then $\P_\ell(\tilde \Psi^{E_0}(2\pi, 0) -\id)=0$, by assumption,   and with \eqref{e.Psi} we get
\begin{align*}
\P_\ell( \tilde\Phi^E(2\pi, 0) -\id)  & =\begin{pmatrix}
                          \cos\left(\alpha_\ell^E/\epsilon\right) & \sin\left(\alpha_\ell^E/\epsilon\right)\\
                          -\sin\left(\alpha_\ell^E/\epsilon\right) & \cos\left(\alpha_\ell^E/\epsilon\right)
                         \end{pmatrix}-\id_{\R^2}\\
                         &={\epsilon}^{-1}{\partial_E \alpha_\ell^{E}|_{E=E_0} (E-E_0)} J+\kO((E-E_0)^2/\epsilon) + \kO(\epsilon^{-1}\mu), 
\end{align*}
where $J$ is the standard $(2,2)$ symplectic matrix.
Due to \eqref{e.degcond} this implies that there is some $c_{\Phi,\ell}>0$ such that for   \begin{equation}\label{e.EBif}
\sqrt\mu \leq |E-E_0| \ll \epsilon
\end{equation}
we have
\[
\|\P_\ell( \tilde\Phi^E(2\pi, 0) -\id_{\R^2})\| \geq c_{\Phi,\ell} |E-E_0|/\epsilon.
\]
Therefore for such $E$ we get 
\[
\|( \tilde\Phi^E(2\pi, 0) -\id)^{-1}\|  = \kO(|E-E_0|^{-1}\epsilon ).
\]
 Then
  \begin{align*}
 \|( \tilde\Phi^E(2\pi, 0) -\id)^{-1}\| =  \kO(\epsilon/\sqrt{\mu}  ).
 \end{align*}
 Due to \eqref{e.tildeRhoE} and \eqref{e.fpEqn} we see that $\tilde\pi_{\mu,E}$ from  \eqref{e.fpEqn} maps
 $\kB_{\tilde\sigma-\kappa}^\kZ(0)$ to itself where $0<\kappa<\tilde\sigma$ provided $\epsilon>0$ is small enough such that $\kO(\sqrt{\mu}) < \tilde\sigma-\kappa$. We use a Cauchy estimate to obtain
 \[
 \|\partial_{\tilde z} \tilde\rho^E\|_{\sigma-\kappa} \leq \|\partial_{\tilde z} \tilde\rho^E\|_{\sigma}/\kappa = \kO(\epsilon^{-1}\mu)
 \]
 uniformly in $E \in (E_1, E_2)$.
  We therefore conclude that  for small enough $\epsilon>0$ the contraction mapping theorem applies to \eqref{e.fpEqn} for energy values $E$ near $E_0$ satisfying \eqref{e.EBif}.
 Next note that due to \eqref{e.degcond} for any fixed $\epsilon>0$ and every $\ell=1,\ldots, d_\kZ$ there  are $\kO(\epsilon^{-1})$ many
 $E$ values such that $\lambda_\ell(E)=1$. Therefore, since $d_\kZ<\infty$, there are $\kO(\epsilon^{-1})$ many
 critical $E$  values; if we exclude an interval of length  $\sqrt{\mu}$ around each of them then we can guarantee
 the persistence of periodic orbits for energy values on a complement of this set.
\end{proof}

\appendix
 
 \section{Exponential accurate slow manifolds in general systems}\label{a.genves}
In this appendix we consider the following system 
\begin{align}
 \dot w = \epsilon W_0(w,z_0),\quad \dot z_0 =Z_0(w,z_0)= r_0(w)+L z_0 + a_0(w)z_0 +F_0(w,z_0).\label{e.sysapp}
\end{align}
 with $(w,z_0)\in (\kV+\i\nu_0)\times (\kS+\i\sigma_0)$ and where $\kV$ and $\kS$ are bounded and open sets in $\kW=\R^{n_\kW}$ and the  Banach space  $\kZ$ respectively, and, as before, $\kS$ is a neighbourhood of $0$ and $\sigma_0, \nu_0>0$. 
 Here the $\dot z_0$ equation is a semilinear evolution equation and the slow  vector field $W_0$ is bounded  as detailed below.
We assume the following: 
\begin{itemize}
 \item[(G0)] 
 $L$ is a densely defined closed operator which either generates a strongly continuous semigroup
 or an analytic semigroup. 
 \end{itemize}
 In the following we set $\alpha=0$ in case $L$ generates a strongly continuous semigroup and assume
 $\alpha \in [0,1)$ otherwise. If $L$ generates an analytical semigroup let $\lambda_0\in \R$ be in the resolvent
 set of $L$  such that
 $\|(\lambda_0+L)^{-1}\|\leq 1$. Note that this is possible because $-L$ is sectorial and so there are some $M_L>0$, $\lambda_L \in \R$  and a sector $S = \{\lambda \in\C; |{\rm arg}(\lambda - \lambda_L)| < \phi \}$, where $\phi<\pi/2$  such that  any $\lambda \in \C\setminus S$ is in the resolvent set of
 $-L$ and satisfies $\|(\lambda +L)^{-1}\| \leq M_L/|\lambda -\lambda_L|$ \cite{Henry}.
 Define the Banach space $\kZ_\alpha:= D((\lambda_0+L)^\alpha)$ 
 with norm 
 \[
  \Wert z \Wert: = \|z\|_{\kZ_\alpha} := \| (\lambda_0+L)^\alpha z\|.
  \] 
 Then by construction $ \Wert z \Wert\geq  \|z\|$.
 We assume that $\kS+\i\sigma_0 \subseteq \kZ_\alpha$. 
 Similarly as before for a map $A \in \kE(\kZ;\kZ_\alpha)$ we define
$\Vert A\Wert : = \|A\|_{\kE(\kZ;\kZ_\alpha)}$, for a map $A\in \kE(\kZ_\alpha;\kZ)$ we define $\Wert A\Vert : = \|A\|_{\kE(\kZ_\alpha;\kZ)}$ and for a map
$F:(\kV+\i\nu_0)\times (\kS+\i \sigma_0)\to \kZ$ we define
\[
\Wert {F} \Vert_{\nu_0,\sigma_0} = \sup_{\substack{w  \in \kV+\i\nu_0\\ w \in \kS+\i \sigma_0}}
\| F(w,z)\|.
\]
Similarly we define   $\Wert {W}_0 \Vert_{\nu_0,\sigma_0}$ for a map
$W:(\kV+\i\nu_0)\times (\kS+\i \sigma_0)\to \kW$.

We further assume the following:
 \begin{itemize}
\item[(G1)]  
The functions  ${r}_0:(\kV+\i\nu_0) \rightarrow \kZ^{\C}$, ${a}_0:(\kV+\i\nu_0) \rightarrow \kE(\kZ^{\C}_\alpha;\kZ^\C)$, ${F}_0:(\kV+\i\nu_0)\times (\kS+\i\sigma_0)\rightarrow \kZ^{\C}$
and 
$W_0:(\kV+\i\nu_0) \times (\kS+\i\sigma_0)  \rightarrow \kW^{\C}$ are real analytic and uniformly bounded by
$\delta_0= \Vert {r}_0 \Vert_{\nu_0}$,
$C_{a_0}=\Wert {a}_0 \Vert_{\nu_0}, $
$C_{F_0}=\Wert {F}_0 \Vert_{\nu_0,\sigma_0}, $ and
$C_{W_0} =\Wert {W}_0 \Vert_{\nu_0,\sigma_0}.$
Here,  as before, $F_0(z_0) = \kO(\Wert z_0\Wert^2)$.
Moreover, $C_{W_0}' =\Wert \partial_z {W}_0 \Vert_{\nu_0,\sigma_0},\,C_{W_0}'' =\Wert \partial_z^2 {W}_0 \Vert_{\nu_0,\sigma_0}$.
\item[(G2)]  The operator
$(L+a_0(\cdot))^{-1}:(\kV+\i\nu_0)\rightarrow \kE(\kZ^{\C},\kZ_\alpha^{\C})$ 
 is real analytic with  
 \[
 \Vert (L+ a_0(\cdot))^{-1}\Wert_{\nu_0} \le \frac{K_0}{2}.
 \]
\end{itemize}
Note that   (G2) is true for some $K_0>0$ if 
if $\Wert a_0-\lambda_0\|_{\nu_0}$ is sufficiently small.
In this setting, by semigroup theory, the semiflow of \eqref{e.sysapp} is well defined, and the
$\dot z$ equation is parabolic if $\alpha>0$, see \cite{Henry, P83}.
 We then have the following result:

\begin{theorem}\label{t.G}
Assume (G0-G2).  Let $\sigma_0>\sigma> 0$, $\nu_0>\nu> 0$.
 Then for  $\delta_0\geq 0$ and $\epsilon>0$ sufficiently small the following holds true: There exists a transformation of the fast variables $z_0=  \zeta(w)+  z$, $(w, z)\in (\kV+\i{\nu})\times (\kS+\i {\sigma})$, $z_0 \in \kS+\i \sigma_0$, with $\Wert   \zeta \Wert_{\nu}=\kO(\epsilon)$ so that
 \begin{align*}
  \dot{ z} =   r(w)+\kO( z),
 \end{align*}
where
\begin{align*}
 \Vert  r(w)\Vert \le C_1 \Vert W_0(w,  \zeta(w))\Vert e^{-C_2/\epsilon}.
\end{align*}
Here  $C_1 $ and  $C_2$ are positive constants which depend solely on $C_{a_0}$,  $K_0$, $C_{F_0}$, $C_\kS$, $\sigma_0$, $\sigma$, $\nu_0$, $\nu$, $C_{W_0}$, $C_{W_0}'$ and $C_{W_0}''$.  
\end{theorem}
In other words: $\{ z=0\}$ is an \textnormal{almost} invariant slow manifold that contains all equilibria of \eqref{e.sysapp} near $\{z_0=0\}$. 
 This result was proved by Neishtadt in the case that $d_\kZ<\infty$ and that $L$ is bounded, using a sligthly different iterative step, as outlined in the introduction, Section \ref{s.intro}. The advantage of MacKay's method which we use in the proof is that the slow manifold we construct contains all nearby equilibria. 
 
For the proof of the Theorem  \ref{t.G} we need the following notion: For $R>0$ such that $\kB_R^{\kZ_\alpha^\C}(0) \subseteq \kS+\i\sigma_0$ we define
$C_F''[\nu_0, R]$ as bound of
\[
 \sup_{\substack{ w \in \kV+\i \nu_0\\ \Wert z\Wert \leq R}}
   \|\partial_z^2 F(w,z)\|_{\kE(\kZ_\alpha\times \kZ_\alpha;\kZ)} \leq C_F''[\nu_0, R]. \]
We need the following modification of Lemma \ref{l.analimpl} which is straightforward to
prove:
\begin{lemma}\label{l.ZetaG}
Assume (G0-G2), with the subscript dropped. Moreover assume that
 \begin{align}
\delta < \min(\sigma/K,  2/( K^2 C_{F}''[\nu, K\delta]).\label{e.deqn0}
\end{align}
 Then
 \begin{equation}\label{e.Zeta}
 0 = r(w) + Lz + a(w) z + F(w,z)
 \end{equation}
 has a locally unique solution $z=\zeta(w)\in {\kZ}_{\alpha}^\C$ satisfying:
\begin{align}
 \Wert \zeta(w)\Wert \le K \Vert r(w) \Vert,\label{e.etax}
\end{align}
for every $w\in \kV+\i\nu$. Moreover  
   $\zeta\in C^\omega(\kV+\i\nu;\kZ^\C_{\alpha}).$ 
\end{lemma}

Next we set up an iterative lemma.  

 \begin{lemma}\label{l.ItGen} (\textbf{The Iterative Lemma})
 Assume (G0-G2) with the subscript dropped and assume \eqref{e.deqn0}.
Let $\zeta=\zeta(w)$ be the solution from Lemma \ref{l.ZetaG}.  Let $\nu_+ = \nu-\xi>0$ and $\sigma_+=\sigma - \xi \ge \kappa>0$.  
If 
\begin{align}
 2 K\max(\epsilon C_W', \delta) \leq \xi \leq \min(\nu,\sigma)
\label{e.CondItLem}
\end{align}
the map $z=\zeta(w)+z_{+}$, $(w,z_+)\in (\kV+\i\nu_+)\times (\kS+\i\sigma_+)$   transforms \eqref{e.sysapp} into
\begin{align}
 \dot w = \epsilon W_+(w,z_+),\quad \dot z_+ = r_+(w)+Lz + a_+(w)z_+ +F_+(w,z_+),\label{e.sysappp}
\end{align}
and there is a constant $c$ which is continuous and increasing in $1/\kappa$, $C_W$, $C_W'$,
$C_W''$, $K$, $C_F$ and $C_\kS=   \Wert z\Wert_\sigma$ 
and depends  on those constants only such that
\begin{subequations}
\begin{align}
 \Vert r_+(w)\Vert \le \frac{K\epsilon  }{\xi}\Vert W_+(w,0)\Vert\delta  & \leq  
 \frac{\epsilon c \delta}{\xi}\quad\mbox{for}\quad w \in \kV+\i\nu_+,\label{e.rp}\\
 \Wert a_+-a\Vert_{\nu_+} &\le c\delta,\label{e.a+}\\
 \Vert (L + a_+(w))^{-1}\Wert_{\nu_+} &\le \frac{K_+}{2}:= \frac{K}{2}+c\delta,\label{e.Kp}\\
\Wert F_+ - F\Vert_{\nu_+,\sigma_+} &\leq \frac{c\delta}{\xi}
\end{align}
\end{subequations}
provided that $c\delta <1$. 
  \end{lemma}
  \bpr
The existence of $\zeta(w)$ follows from Lemma \ref{l.ZetaG}  and yields
\begin{align}
W_+(w,z_+) &= W(w,\zeta+z_+),\nonumber\\
r_+(w) &= -\epsilon \partial_w \zeta(w) W(w,\zeta(w)), \label{e.rho+}\\
a_+(w)&= a(w)-\epsilon  \partial_w \zeta(w) \partial_z W(w,\zeta(w))+\partial_z F(w,\zeta(w)),\nonumber
\end{align}
and 
\begin{align}
F_+(w,z_+) &= -\epsilon\partial_w \zeta(w) \int_0^1 (1-s) \partial_z^2 {W}(w,\zeta(w)+sz_+)\d s z_+^2+\tilde F(w,z_+),\label{e.R1eqn}
\end{align}
where
\begin{align}
\tilde F(w,z_+)&=F(w,\zeta(w)+z_+)-F(w,\zeta(w))-\partial_z F(w,\zeta(w))z_+.\label{e.Rtilde}
\end{align}
Using Lemma \ref{l.ZetaG} and a Cauchy estimate give
\begin{align}
 \Wert \partial_w \zeta \Wert_{\nu-\xi} \le \frac{\Wert \zeta \Wert_{\nu}}{\xi} \leq 
 \frac{K\delta}{\xi},\label{e.etax2}
\end{align}
for $\nu-\xi>0$ and $\xi>0$. Hence \eqref{e.rho+} directly gives \eqref{e.rp}.
To estimate $a_+-a$ we first
note that {by \eqref{e.etax2} 
\[
\Wert \partial_w \zeta \partial_z W(w,\zeta)\Vert_{\nu-\xi}\le   C_W' \frac{\Wert \zeta\Wert_{\nu}}{\xi}
\]
for $w\in \kV+\i(\nu-\xi)$}.
Moreover, since $F$ is quadratic, by a Cauchy estimate,
\begin{align*}
 \Wert \partial_z F(w,\zeta)\Vert 
 &= \Wert \int_0^1\partial_z^2 F(w,{s}\zeta)\zeta{\d s} \Vert\le {2C_F}{\kappa^{-2}}\Wert \zeta\Wert_{\nu},
\end{align*}
for all $w\in \kV+\i(\nu-\xi)$, 
and therefore for all such $w$,
as $\Wert\zeta\Wert_\nu \leq K \delta$ by Lemma \ref{l.ZetaG}, 
\begin{align*}
 \Wert a_+-a \Vert_{\nu-\xi}&\le C_W' \frac{\epsilon K}{\xi}\delta  + 2C_F {\kappa^{-2}}K\delta \le \left(\frac12 + 2C_F {\kappa^{-2}}K\right)\delta = c \delta,
\end{align*}
 where we have used  \eqref{e.CondItLem}.
This proves \eqref{e.a+}.

From
\[
(L+a_+)^{-1} = (L+a)^{-1}(\id + (a_+-a)(L+a)^{-1})^{-1}
\]
we get
\begin{align*}
\Vert (L+a_+(w))^{-1}\Wert &\le 
\Vert  (L+a)^{-1}\Wert( 1- \Vert (a_+-a)(L+a)^{-1}\Vert)^{-1}\\
&  \leq
\frac{K}{2}\left(1 -\frac{K}{2} c \delta\right)^{-1} \le \frac{K}{2}(1 +cK\delta).
\end{align*}
for $K c \delta<1$.  Here we have used that $(1-x)^{-1}\le 1+2x$ if $0\le x \le \frac12$
with $x = \frac{K}{2} c \delta$. Redefining $c$ to $\max(1,K^2)c$ proves \eqref{e.Kp} for $c\delta <1$.

For $F_+-F$, {with $F_+$ from} \eqref{e.R1eqn} we first estimate $\tilde F-F$ {from} \eqref{e.Rtilde}. {This gives:}
\begin{align*}
\Wert \tilde F-F\Vert_{\nu-\xi,\sigma-\xi}
& \le \Wert F(w,\zeta+z_+)-F(w,z_+)\Vert_{\nu-\xi,\sigma-\xi} 
 +\Wert F(w,\zeta)\Vert_{\nu-\xi} \nonumber\\
 & +\Wert \partial_z F(w,\zeta)z_+\Vert_{\nu-\xi,\sigma-\xi}\nonumber\\
 &\le   \Wert \partial_z F\Vert_{\nu,\sigma-\xi/2} \Wert \zeta\Wert_{\nu}+
 \Vert \int_0^1(1-t)\partial_z^2 F(w,t\zeta(w))\zeta^2(w) \d t\Vert_{\nu-\xi} \\
 &+\Wert \int_0^1\partial_z^2 F(w,t\zeta(w))\zeta(w) z_+ \d t\Vert_{\nu-\xi,\sigma-\xi}\nonumber\\
&\le \Wert \partial_z F\Vert_{\nu,\sigma-\xi/2}\Wert \zeta\Wert_{\nu} +\frac12 \Wert \partial_z^2 F\Vert_{\nu,\kappa} \Wert \zeta\Wert_{\nu}^2+ \Wert \partial_z^2 F\Vert_{\nu,\kappa} \Wert \zeta \Wert_{\nu} C_\kS\\
 &\le C_F\left(\frac{2K\delta}{\xi}+{\kappa^{-2}}K\delta\left(K\delta+ 2C_\kS\right)\right).\nonumber
\end{align*}
 Here we have used that $\Wert\zeta\Wert_{\nu} \leq K \delta$ by Lemma \ref{l.ZetaG} and that $\xi\ge 2K\delta$, see \eqref{e.CondItLem}. Therefore
\begin{align*}
\Wert  F_+-F \Vert_{\nu-\xi,\sigma-\xi} &\le
 \frac12 \epsilon \Vert \partial_w\zeta \Vert_{\nu-\xi} \Vert \partial_z^2 W\Vert_{\nu,\sigma-\xi/2}C_\kS^2 
 + \Wert \tilde F-F\Vert_{\nu-\xi,\sigma-\xi}\\
&\le \frac{\epsilon K C_\kS^2\delta C_W''}{2\xi}
+C_F\left(\frac{2K\delta}{\xi}+{\kappa^{-2}}K\delta\left(K\delta+ 2{C_\kS}\right)\right)
\leq \frac{c\delta}{\xi}
\end{align*}
for a suitable choice of $c$ with the given properties. Here we used \eqref{e.CondItLem}.
  \epr
  
\bigskip
\noindent
{\bf Proof of Theorem \ref{t.G}}.
Let  $\xi_0=\xi_1= \frac14 \min\{\nu_0- {\nu},\sigma_0- {\sigma}\}$. For sufficiently small $\delta_0$ and $\epsilon$ we can satisfy
the conditions  \eqref{e.CondItLem} and \eqref{e.deqn0} of the iterative Lemma
 \ref{l.ItGen} applied to \eqref{e.sysapp} with $\xi_0$ as above to obtain that $\delta_1 = \kO(\epsilon)$.  For all successive iterations we use the bound
$C_{F_n}''[\nu_n,K_n\delta_n] = 2 C_{F_n}/\kappa^2$ for $\sigma_n-\kappa>K_n \delta_n$ with $\sigma>\kappa>0$ which changes
 \eqref{e.deqn0} to  the condition 
\begin{equation}\label{e.deltaKappaG}
\delta_n < \min(\sigma_n/K_n,\kappa^2 K_n^{-2}C_{F_n}^{-1}).
\end{equation}
 Here, as before
$K_n$, $\delta_n$ etc.~denote the constants of Lemma   \ref{l.ItGen} after $n$ iterations, i.e., 
$K_n/2$ is  the upper bound in \eqref{e.Kp} of the operator 
$(L + a_n)^{-1}$ on $\kV+\i\nu_n$ where $\nu_n=\nu_2-\sum_{k=1}^{n-1}\xi_k>0$. Furthermore   $C_{F_n}$ is the norm of $F_n$ on $(\kV+\i\nu_n)\times (\kS+\i\sigma_n)$ where $\sigma_n=\sigma_2-\sum_{k=1}^{n-1} \xi_k>0$
and  $C_{a_n}$ is defined analogously.   

Since $\delta_1 = \kO(\epsilon)$ we can satisfy \eqref{e.CondItLem} and \eqref{e.deltaKappaG}   by choosing for sufficiently small $\epsilon$ and therefore can apply Lemma  \ref{l.ItGen} again to obtain  $\delta_2 = \kO(\epsilon^2)$, cf. \eqref{e.rp}.

Applying   the Iterative Lemma \ref{l.ItGen} successively we have
\begin{equation}\label{e.ItGen}
\max(C_{a_{n+1}} - C_{a_n}, \tfrac12( K_{n+1}-K_n), C_{F_{n+1}}-C_{F_n}) \leq c_n\delta_n/\xi_n,
\end{equation}
and
\[
\delta_{n+1} \leq c_n \delta_n \epsilon/ \xi_n,
\] where we choose $\xi_n\leq 1$.
Taking $\xi_n\geq 2c_n\epsilon>0$ we get 
\begin{align*}
 \delta_{n+1}\le 2^{-1}\delta_n\le 2^{-n}\delta_2 \leq  2^{-n}  C \epsilon^2.
\end{align*}
Then from \eqref{e.ItGen} we get
\[
\max(\tfrac12(K_n- K_2),  C_{a_n}-C_{a_2},  C_{F_n} - C_{F_2}) \leq  \sum_{n=2}^N \delta_n /\epsilon
\leq 2 C\epsilon.
\] by the geometric series formula.
Since $c_n$ is a continuous and increasing  function of $C_{F_n}$, $C_{a_n}$, $K_n$ and $1/\kappa$ and
these are bounded in $n$, there is some $c_*$ such that $c_n \leq c_*$ for all $n=1,\ldots, N$. We then set $\xi_n = 2 c_*\epsilon$ for $n=2,\ldots, N$.
From the requirement 
\[
 \nu\leq \nu_2 - 2 c_*(N-1_\epsilon = \nu_{N+1}, \quad \sigma \leq  \sigma_2 - 2c_*(N-1)\epsilon = \sigma_{N+1}
\]
we can take $N$ to be as $N= \lceil\frac{M}{4c_*\epsilon}\rceil$ where $M=\min(\nu_0-\nu, \sigma_0-\sigma)$
which concludes the proof.  
\qed

\medskip
The following example shows that the error $\|r_n(w)\|$ 
 in the slow manifold
after $n$ steps
does not behave like  $\kO((\epsilon \Vert W(w,\cdot)\Vert)^n)$
in general, as conjectured by MacKay \cite{mac1}, see the discussion in the introduction.

\begin{example}\label{ex.e1}\rm
We consider the simple linear, {two-dimensional} example:
\begin{align}\label{e.counterEx}
 \dot w & ={\epsilon W(w,z)=} \epsilon w,\,\quad \dot z ={Z(w,z)=} \epsilon w - z.
\end{align}
Here $z=0$ is actually normally hyperbolic and there is an invariant slow manifold nearby:
\begin{align}
 z = \frac{\epsilon}{1+\epsilon}w.\label{zInv}
\end{align}
Notice that $(w,z)=(0,0)$ is an equilibrium (saddle for $\epsilon>0$). 
 Applying MacKay's method $n$ times to this example gives
 \begin{align}
 z_{n} = 0 \quad\mbox{where}\quad  z = z_n +\sum_{k=0}^n \zeta_k(w) = \sum_{k=1}^n (-1)^{k}  \epsilon^{k+1} w,\label{zK}
 \end{align}
as an approximately invariant slow manifold. In this case the approximation \eqref{zK} also coincide with the $n$th degree Taylor polynomial of \eqref{zInv}. The error field is 
$r_n(w) = \zeta_n(w) = (-1)^{n}   \epsilon^{n+1} w$ 
which directly illustrates why MacKay's conjecture is incorrect. For a nonlinear example, one may replace $W(w,z)=w$ by $\sW (w)$ satisfying $\sW(0)=0,\,\sW'(0)\ne 0$. Then $(w,z)=(0,0)$ is still a hyperbolic equilibrium, and we have $r_1(w) = -\epsilon^2 \sW(w)$ and
$r_2(w) = \epsilon^3 \sW'(w) \sW(w)$ which cannot be bounded above  from above by an expression with   $\vert \sW(w)\vert^2$ as a factor. 
\end{example}




\bibliographystyle{plain}

\newpage

\end{document}